\documentclass[reqno]{amsart}
\usepackage{color}
\usepackage{hyperref}
\usepackage{citehack}
\usepackage{esint}
\allowdisplaybreaks

\newcommand{\R}{\mathbb{R}}
\newcommand{\N}{\mathbb{N}}

\newcommand{\CO}{C^{1,\alpha}(\partial\Omega)}
\newcommand{\COo}{C^{1,\alpha}(\partial\Omega^o)}
\newcommand{\COi}{C^{1,\alpha}(\partial\Omega^i)}
\newcommand{\COh}{C^{1,\alpha}(\partial\Omega^h)}
\newcommand{\COharm}{C^{1,\alpha}_{\mathrm{harm}}(\overline{\Omega})}

\newcommand{\COiharm}{C^{1,\alpha}_{\mathrm{harm}}(\overline{\Omega^i})}
\newcommand{\COeiharm}{C^{1,\alpha}_{\mathrm{harm}}(\epsilon \overline{\Omega^i})}
\newcommand{\COoiharm}{C^{1,\alpha}_{\mathrm{harm}}(\overline{\Omega^o} \setminus \epsilon \Omega^i)}

\begin{document}

\newtheorem{teo}{Theorem}[section]
\newtheorem{defin}[teo]{Definition}
\newtheorem{rem}[teo]{Remark}
\newtheorem{prop}[teo]{Proposition}
\newtheorem{lemma}[teo]{Lemma}
\newtheorem{cor}[teo]{Corollary}
\newtheorem{fact}[teo]{Fact}

\title[A perturbed nonlinear non-autonomous transmission problem]{Existence of solutions for a singularly perturbed nonlinear non-autonomous transmission problem}

\author{Riccardo Molinarolo}

\maketitle

\noindent
{\bf Abstract:}  In this paper we analyse a boundary value problem for the Laplace equation with a nonlinear non-autonomous transmission conditions on the boundary of a small inclusion of size $\epsilon$. We show that the problem has solutions for $\epsilon$ small enough and we investigate the dependence of a specific family of solutions upon $\epsilon$. By adopting a functional analytic approach we prove that the map which takes $\epsilon$ to (suitable restrictions of) the corresponding solution  can be represented in terms of real analytic functions.
\vspace{9pt}

 \par
\vspace{9pt}


\section{Introduction}

This paper is devoted to the analysis of a singularly perturbed nonlinear transmission problem for the Laplace equation in the pair of sets composed by a perforated domain and a (small) inclusion. The study of the behaviour of the solutions of boundary value problems in domain with small holes or inclusions has attracted the attention of several pure and applied mathematicians and it is impossible to provide a complete list of contributions. 
From an application point of view, boundary value problems in domains with small holes or inclusions can be the mathematical model of the heat conduction in bodies with small cavities and impurities and thus they are extensively studied in the theory of dilute composite materials (cf.~Movchan, Movchan, and Poulton \cite{MoMoPo02}).
In particular, transmission conditions like the ones that we study in this paper can be analytically derived in the case of a thin reactive heat conducting interphase situated between two different materials (see the works of Mishuris, Miszuris and \"Ochsner \cite{MiMiOc07}, \cite{MiMiOc08}, and of Miszuris and \"Ochsner \cite{MiOc13} and the references therein). Moreover, we point out that nonlinear transmission conditions arise also in the framework of elasto-plastic material (see e.g.~Miszuris and \"Ochsner \cite{MiOc05} and Mishuris, Miszuris, \"Ochsner, and Piccolroaz \cite{MiMiOcPi14}) and in the framework of articular cartilage problems (cf.~Vitucci, Argatov, and Mishuris \cite{ViArMi17}).

In order to introduce our specific problem, we begin by presenting the geometric framework. We fix once for all a natural number 
\[
n\ge 3
\]
which will be  the dimension of the Euclidean  space $\mathbb{R}^n$ we are going to work in, and a parameter
\[
\alpha \in ]0,1[
\]
which we use to define the regularity of our sets and functions. We observe that the case of dimension $n=3$ has physical relevance. However, the analysis for $n=3$ and for $n\ge 3$ is much the same and for this reason we opt for  the more general setting. Instead, the case of dimension $n=2$ requires  specific techniques.

Next, we introduce two sets $\Omega^o$ and $\Omega^i$ such that 
\[
\begin{split}
&\mbox{$\Omega^o$, $\Omega^i$ are bounded open connected subsets of $\R^n$ of class $C^{1,\alpha}$,} 
\\
&\mbox{with exteriors  $\R^n\setminus \overline{\Omega^o}$ and $\R^n\setminus \overline{\Omega^i}$ connected and the origin  $0$}
\\
&\mbox{of $\mathbb{R}^n$ belongs both to $\Omega^o$ and to $\Omega^i$.}
\end{split}
\]
Here the superscript ``$o$" stands for ``outer domain" whereas the superscript ``$i$"  stands for ``inner domain". We set
\[
\epsilon_0 \equiv \mbox{sup}\{\theta \in ]0, 1[: \epsilon \overline{\Omega^i} \subseteq \Omega^o, \ \forall \epsilon \in ]- \theta, \theta[  \} 
\]
and we define the perforated domain $\Omega(\epsilon)$ by setting
\[
\Omega(\epsilon) \equiv \Omega^o \setminus \epsilon \overline{\Omega^i}
\]
for all $\epsilon\in]-\epsilon_0,\epsilon_0[$. Then we fix three functions
\[
F \colon ]-\epsilon_0,\epsilon_0[ \times \partial \Omega ^i \times \R \to \R\,,\qquad G \colon ]-\epsilon_0,\epsilon_0[ \times \partial \Omega ^i \times \R \to \R\,,\text{ and }\quad f^o \in \COo,
\]
and, for $\epsilon \in ]0,\epsilon_0[$, we consider the following nonlinear non-autonomous transmission problem in the perforated domain $\Omega(\epsilon)$ for a pair of functions $(u^o,u^i)\in C^{1,\alpha}(\overline{\Omega(\epsilon)})\times C^{1,\alpha}(\overline{\epsilon\Omega^i})$:
\begin{equation}\label{princeq}
\begin{cases}
\Delta u^o = 0 & \mbox{in } \Omega(\epsilon), \\
\Delta u^i = 0 & \mbox{in } \epsilon \Omega^i, \\
u^o(x)=f^o(x) & \forall x \in \partial \Omega^o, \\
u^o(x) = F\left(\epsilon,\frac{x}{\epsilon},u^i(x)\right) & \forall x \in \epsilon \partial \Omega^i, \\
\nu_{\epsilon \Omega^i} \cdot \nabla u^o (x) - \nu_{\epsilon \Omega^i} \cdot \nabla u^i (x) = G\left(\epsilon, \frac{x}{\epsilon},u^i(x)\right) & \forall x \in \epsilon \partial \Omega^i.
\end{cases}
\end{equation}
Here $\nu_{\epsilon\Omega^i}$ denotes the outer exterior normal to $\epsilon\Omega^i$.

Problem \eqref{princeq} may, for example, model the heat conduction in a (nonlinear) composite material. Indeed, $u^o$ and $u^i$ could represent the temperature distribution in $\Omega(\epsilon)$ and in the inclusion $\epsilon \Omega^i$, respectively. The third condition in \eqref{princeq} means that we are prescribing the temperature distribution on the exterior boundary $\partial \Omega^o$. The fourth condition says that on the interface $\epsilon \partial \Omega^i$ the temperature distribution $u^o$ depends nonlinearly on the size of the inclusion, on the position on the interface, and on the temperature distribution $u^i$. The fifth condition, instead, says that the jump of the heat flux on the interface depends nonlinearly on the size of the inclusion, on the position on the interface, and on the temperature distribution $u^i$.

Since problem \eqref{princeq} is nonlinear, one cannot, {\it a priori}, claim that it has a solution. As a first result, we will prove that under suitable conditions on $F$ and $G$ and possibly shrinking $\epsilon_0$, problem \eqref{princeq} has a solution $(u^o_\epsilon,u^i_\epsilon) \in C^{1,\alpha}(\overline{\Omega(\epsilon)}) \times C^{1,\alpha}(\overline{\epsilon\Omega^i})$ for all $\epsilon \in ]0,\epsilon_0[$.
Then, we will turn to analyse the asymptotic behavior of the family of solutions $\{(u^o_\epsilon,u^i_\epsilon)\}_{\epsilon\in]0,\epsilon_0[}$ as $\epsilon$ approaches the degenerate value $0$. 

In literature, one of the most used approach to do that would be to write out an asymptotic expansion of $u^o_\epsilon$ and $u^i_\epsilon$. Asymptotic expansion techniques for singularly perturbed linear transmission problems have been exploited by several authors: here we mention the works of Ammari and collaborators \cite{AmGaKaLe, AmGaGiJiSe16, AmKa07, AmKaKi05, AmKaLi06, AmKaTo05}, Maz'ya, Movchan, and Nieves \cite{MaMoNi10}, Nieves \cite{Ni17}, Novotny and Soko\l owski \cite{NoSo13}, and, in particular, concerning nonlinear problems, Iguernane, Nazarov, Roche, Soko\l owski, and  Szulc \cite{IgNaRoSoSz09}.
Cluster of holes have been also considered in the work of Bonnaillie-No\"el, Dambrine, Tordeux, and Vial \cite{BoDaToVi09}, Bonnaillie-No\"el and Dambrine \cite{BoDa13}, and Bonnaillie-No\"el, Dambrine, and Lacave \cite{BoDaLa}.
Moreover, functional equation methods for the analysis of linear and nonlinear transmission problems in domains with circular inclusions have been applied, for example, in Castro,  Kapanadze, and Pesetskaya \cite{CaKaPe15}, Kapanadze,  Mishuris,  and  Pesetskaya \cite{KaMiPe15, KaMiPe15b}, Kapanadze, Miszuris, and Pesetskaya \cite{KaMiPe16}.
We mention that potential theoretic techniques have been widely exploited to study nonlinear boundary value problems with  transmission conditions  by  Berger, Warnecke, and Wendland \cite{BeWaWe90}, by Costabel and Stephan \cite{CoSt90}, by Gatica and Hsiao \cite{GaHs95}, and by Barrenechea and Gatica \cite{BaGa96}, and that boundary integral methods have been applied also by Mityushev and Rogosin for the analysis of transmission problems in the two dimensional plane (cf.~\cite[Chap.~5]{MiRo00}).
Finally, we point out that problems with small holes or inclusions have been analysed also from the numerical point of view, for example in the works of Chesnel and Claeys \cite{ChCl16}  and of Babu\v{s}ka, Soane,  and Suri \cite{BaSoSu17}.
We also mention the works of Mishuris, Miszuris and \"Ochsner \cite{MiMiOc07}, \cite{MiMiOc08}, and of Miszuris and \"Ochsner \cite{MiOc13}, in which transmission conditions are numerically tested with simulation based on the finite element method.

In this paper, instead, we adopt the functional analytic approach proposed by Lanza de Cristoforis for the analysis of singular perturbation problems in perforated domain (cf. Lanza de Cristoforis \cite{La02}). We prove that (suitable restriction of) the family of solutions $\{(u^o_\epsilon,u^i_\epsilon)\}_{\epsilon\in]0,\epsilon_0[}$ can be described in terms of real analytic functions of $\epsilon$. Such a result implies the possibility of expanding the solutions in terms of convergent power series of the singular perturbation parameters. We note that the functional analytic approach has been first applied to study boundary value problems in a domain with a small hole confined in the interior (cf. Lanza de Cristoforis \cite{La08}) and then has been extended to more involved geometric configurations, such as moderately close holes (cf. Dalla Riva and Musolino \cite{DaMu16, DaMu17}), holes approaching to the boundary (cf. Bonnaillie-No\"el, Dalla Riva, Dambrine, and Musolino \cite{BoDaDaMu}), and  perturbations close to the vertex of a sector (cf. Costabel, Dalla Riva, Dauge, and Musolino \cite{CoDaDaMu17}). Moreover, as in this paper, it has been applied also to boundary value problems with nonlinear conditions as in Lanza de Cristoforis \cite{La07} and Dalla Riva and Lanza de Cristoforis \cite{DaLa10}.

We observe that a similar problem, but with homogeneous contact conditions (i.e.~with $F$ and $G$ which do not depend on the position on the contact boundary) has been studied by Lanza de Cristoforis in \cite{La10} for a bounded domain with a small hole and in Lanza de Cristoforis and Musolino \cite{LaMu14} in the periodic setting. We also mention the work of Dalla Riva and Mishuris \cite{DaMi15}, where the existence of solutions for problem \eqref{princeq} has been investigated in the case of a ``big'' inclusion with a method based on the Schauder fixed point theorem. As we shall see, the analysis of this paper is instead based on the implicit function theorem.

We briefly summarize our strategy. We first introduce a suitable representation of a solution of problem \eqref{princeq} in terms of layer potentials with unknown densities. Then, by an appropriate change of variables and by exploiting the Taylor expansion of certain terms, we convert problem \eqref{princeq} into a system of nonlinear integral equations on the boundaries of $\Omega^o$ and $\Omega^i$.  The new system is constructed in such a way that we can use the implicit function theorem to analyse its solution around the degenerate case when $\epsilon=0$. In such a way,  we find the unknown densities as implicit functions and we deduce that they depend real analytically on $\epsilon$. Finally, we exploit again the integral representation of the solutions to prove the existence of $u^o_\epsilon$ and $u^i_\epsilon$ and to analyse their dependence on $\epsilon$. 

This paper is organised as follows. Section \ref{notation} is a section of preliminaries. In Section \ref{classresult} we introduce some classical notions and results of potential theory. In Section \ref{preliminaryresults} we prove two technical lemmas that we exploit in Section \ref{formulation} to convert problem \eqref{princeq} into a system of nonlinear integral equations. In Section \ref{limitingsystem} we analyse the limiting system, i.e.~the system obtained for $\epsilon=0$. In Section \ref{applicationIFT} we apply the implicit function theorem to obtain a real analytic continuation result for the unknown densities of the integral equation system. In Section \ref{analrapprfamsol} we state our main Theorem \ref{uesol} where we show an existence result for $u^o_\epsilon$ and $u^i_\epsilon$ and we analyse their dependence on $\epsilon$ in Theorem \ref{uanal}.

\section{Notation}\label{notation}
We denote the norm of a real normed space $X$ by $\| \cdot \| _X$. We denote by $I_X$ the identity operator from $X$ to itself and we omit the subscript $X$ where no ambiguity can occur.  For $x \in X$ and $R>0$, we denote by $B_X(x,R) \equiv \{y\in X : \|y-x\|_X < R \}$, when $X=\R^d$, $d \in \N\setminus\{0,1\}$, we simply write $B_d(x,R)$ and when $X = \R$ we write $B(x,R)$. If $X$ and $Y$ are normed spaces we endow the product space $X \times Y$ with the norm defined by $\| (x,y) \|_{X \times Y} \equiv \|x\|_X + \|y\|_Y $ for all $(x,y) \in X \times Y$, while we use the Euclidean norm for $\R^d$, $d\in\mathbb{N}\setminus\{0,1\}$. We denote by $\mathcal{L}(X,Y)$ the Banach space of linear and continuous map of $X$ to $Y$, equipped with its usual norm of the uniform convergence on the unit sphere of $X$. If $U$ is an open subset of $X$, and $F:U \to Y$ is a Fr\'echet-differentiable map in $U$, we denote the differential of $F$ by $dF$. Higher order differentials are denoted by $d^m F$, $m \in \N\setminus\{0,1\}$. 
The inverse function of an invertible function $f$ is denoted by $f^{(-1)}$, while the reciprocal of a function $g$ or the inverse of an invertible matrix $A$ are denoted by $g^{-1}$ and $A^{-1}$ respectively. Let $\Omega \subseteq \R^n$. Then $\overline{\Omega}$ denotes the closure of $\Omega$ in $\R^n$, $\partial \Omega$ denotes the boundary of $\Omega$, and $\nu_\Omega$ denotes the outward unit normal to $\partial \Omega$.
For $x \in \R^d$, $x_j$ denotes the $j$-th coordinate of $x$, $|x|$ denotes the Euclidean modulus of $x$ in $\R^d$.
Let $\Omega$ be an open subset of $\R^n$ and $m \in \N \setminus \{0\}$. The space of $m$ times continuously differentiable real-valued function on $\Omega$ is denoted by $C^m(\Omega,\R)$ or more simply by $C^m(\Omega)$ . Let $r \in \N \setminus \{0\}$, $f \in (C^m(\Omega))^r$. The $s$-th component of $f$ is denoted by $f_s$ and the gradient matrix of $f$ is denoted by $\nabla f$. Let $\eta=(\eta_1, \dots ,\eta_n) \in \N^n$ and $|\eta|=\eta_1+ \dots+\eta_n$. Then $D^\eta f \equiv \frac{\partial^{|\eta|}f}{\partial x^{\eta_1}_1, \dots , \partial x^{\eta_n}_n}$. If $r=1$, the Hessian matrix of the second-order partial derivatives of $f$ is denoted by $D^2 f$. The subspace of $C^m(\Omega)$ of those functions $f$ such that $f$ and its derivatives $D^\eta f$ of order $|\eta|\le m$ can be extended with continuity to $\overline{\Omega}$ is denoted $C^m(\overline{\Omega})$. The subspace of $C^m(\overline{\Omega})$ whose functions have $m$-the order derivatives that are H\"{o}lder continuous with exponent $\alpha \in ]0,1[$ is denoted $C^{m,\alpha}(\overline{\Omega})$.  If $f \in C^{0,\alpha}(\overline{\Omega})$, then its H\"{o}lder constant is defined as $|f : \Omega|_\alpha\equiv \mbox{sup} \left\{\frac{|f(x)-f(y)|}{|x-y|^\alpha} : x,y \in \overline{\Omega}, x \neq y \right\}$. The space $C^{m,\alpha}(\overline{\Omega})$, equipped with its usual norm $\|f\|_{C^{m,\alpha}(\overline{\Omega})} \equiv \|f\|_{C^{m}(\overline{\Omega})} + \sum_{|\eta|=m}{|D^\eta f : \Omega|_\alpha}$, is well know to be a Banach space. 
We denote by $C^{m,\alpha}_{\mathrm{loc}}(\R^n \setminus \Omega)$ the space of functions on $\R^n \setminus \Omega$ whose restriction to $\overline{U}$ belongs to $C^{m,\alpha}(\overline{U})$ for all open bounded subsets $U$ of $\R^n \setminus \Omega$. On $C^{m,\alpha}_{\mathrm{loc}}(\R^n \setminus \Omega)$ we consider the natural structure of Fr\'echet space. Finally we set
\begin{equation*}
	C^{m,\alpha}_{\mathrm{harm}}(\overline{\Omega}) \equiv \{ u \in C^{m,\alpha}(\overline{\Omega}) \cap C^2(\Omega): \Delta u = 0 \text{ in } \Omega \}.
\end{equation*}
We say that a bounded open subset of $\R^n$ is of class $C^{m,\alpha}$ if it is a manifold with boundary imbedded in $\R^n$ of class $C^{m,\alpha}$. In particular if $\Omega$ is a $C^{1,\alpha}$ subset of $\R^n$, then $\partial\Omega$ is a $C^{1,\alpha}$ sub-manifold of $\R^n$ of co-dimension $1$.  
If $M$ is a $C^{m,\alpha}$ sub-manifold of $\R^n$ of dimension $d\ge 1$, we define the space $C^{m,\alpha}(M)$ by exploiting a finite local parametrization. Namely, we take a finite open covering $\mathcal{U}_1, \dots, \mathcal{U}_k$ of $M$ and  $C^{m,\alpha}$ local parametrization maps $\gamma_l : \overline{B_{d}(0,1)} \to \overline{\mathcal{U}_l}$ with $l=1,\dots, k$ and we say that $\phi\in C^{m,\alpha}(M)$ if and only if $\phi\circ\gamma_l\in C^{m,\alpha}(\overline{B_{d}(0,1)})$ for all $l=1,\dots, k$. Then for all $\phi\in C^{m,\alpha}(M)$ we define
\[
\|\phi\|_{C^{m,\alpha}(M)} \equiv \sum_{l=1}^k \|\phi\circ\gamma_l \|_{C^{m,\alpha}(\overline{B_{d}(0,1)})}\,.
\]
One verifies that different $C^{m,\alpha}$ finite atlases define the same space  $C^{m,\alpha}(M)$ and equivalent norms on it.
We retain the standard notion for the Lebesgue spaces $L^p$, $p\ge 1$. If $\partial\Omega$ is measurable then we denote by $d\sigma$ the area element on $\partial\Omega$. If $Z$ is a subspace of $L^1(\partial \Omega)$, then we set
\begin{equation*}
Z_0 \equiv \left\{ f \in Z : \int_{\partial\Omega} f \,d\sigma = 0 \right\}.
\end{equation*}

\section{Classical notions of potential theory}\label{classresult}
For the proofs of the results of this section we refer to Folland \cite{Fo95}, Gilbarg and Trudinger \cite{GiTr83}, Schauder \cite{Sc31}, and to the references therein.
\begin{defin}
	We denote by $S_n$ the function from $\R^n \setminus \{0\}$ to $\R$ defined by
\[
	S_n(x) \equiv	\frac{|x|^{2-n}}{(2-n) s_n} \qquad \forall x \in \R^n \setminus \{0\} 
\]
	where $s_n$ denotes the $(n-1)$-dimensional measure of $\partial B_n(0,1)$.
\end{defin}    

$S_n$ is well known to be a fundamental solution of the Laplace operator.
Now let $\Omega$ be an open bounded subset of $\R^n$ of class $C^{1,\alpha}$. We define
\begin{equation*}
\Omega^- \equiv \R^n \setminus \overline{\Omega}.
\end{equation*}

\begin{defin}
	We denote by $v_{\Omega}[\mu]$ the single layer potential with density $\mu$ given by
	\begin{equation*}
	v_{\Omega}[\mu](x) \equiv \int_{\partial \Omega}{S_n(x-y) \mu(y) \,d\sigma_y} \qquad \forall x \in \R^n
	\end{equation*}
	for all $\mu \in C^{0,\alpha}(\partial\Omega)$.
	
	We denote by $w_{\Omega}[\mu]$ the double layer potential with density $\mu$ given by
	\begin{equation*}
	w_{\Omega}[\mu](x) \equiv - \int_{\partial \Omega}{\nu_\Omega(y) \cdot \nabla S_n(x-y) \mu(y) \,d\sigma_y} \qquad \forall x \in \R^n.	
	\end{equation*}
	for all $\mu \in \CO$.
\end{defin}
It is well known that, if $\mu \in C^{0,\alpha}(\partial\Omega)$, then $v_{\Omega}[\mu] \in C^0(\R^n)$. We set 
\begin{equation*}
v^+_{\Omega}[\mu]\equiv v_{\Omega}[\mu]_{| \overline{\Omega}}, \qquad v^-_{\Omega}[\mu] \equiv v_{\Omega}[\mu]_{| \overline{\Omega^-}}.
\end{equation*}
Moreover, if $\mu \in \CO$, then $w_{\Omega}[\mu]_{| \Omega}$ admits a unique continuous extension to $\overline{\Omega}$, which we denote by $w^+_{\Omega}[\mu]$, and $w_{\Omega}[\mu]_{| \Omega^-}$ admits a unique continuous extension to $\overline{\Omega^-}$, which we denote by $w^-_{\Omega}[\mu]$.

\begin{defin}
	We denote by $W_{\partial\Omega}[\mu]$ the boundary integral operator defined by
	\[
	W_{\partial\Omega}[\mu](x) \equiv - \int_{\partial \Omega}{\nu_\Omega(y) \cdot \nabla S_n(x-y) \mu(y) \,d\sigma_y} \qquad \forall x \in \partial \Omega
	\]
	for all $\mu \in \CO$.
	
	We denote by $V_\Omega$ the operator from $C^{0,\alpha}(\partial\Omega)$ to $\CO$ which takes $\mu$ to $V_\Omega[\mu]$ defined by
	\[
	V_{\Omega}[\mu]\equiv v_{\Omega}[\mu]_{|\partial\Omega}.
	\]
	
	We denote by $W^\ast_{\partial\Omega}[\mu]$ the boundary integral operator defined by
	\[
	W^\ast_{\partial\Omega}[\mu](x) \equiv \int_{\partial \Omega}{\nu_\Omega(x) \cdot \nabla S_n(x-y) \mu(y) \,d\sigma_y} \qquad \forall x \in \partial \Omega 
	\]
	for all $\mu \in C^{0,\alpha}(\partial\Omega)$.
\end{defin}

One verifies that $W_{\partial\Omega}: \CO \to \CO$ and $W^\ast_{\partial\Omega}: C^{0,\alpha}(\partial\Omega) \to C^{0,\alpha}(\partial\Omega)$ are transpose one to the other with respect to the duality of $\CO \times C^{0,\alpha}(\partial\Omega)$ induced by the inner product of $L^2(\partial\Omega)$.

In the following Theorem \ref{sdp} we summarize some classical results of potential theory.

\begin{teo}[\bf{Property of single and double layer potentials}]\label{sdp}
	The following statements hold.
	\begin{enumerate}
		\item[(i)] For all $\mu \in C^{0,\alpha}(\partial\Omega)$, the function $v_{\Omega}[\mu]$ is harmonic in $\R^n\setminus \partial\Omega$ and at infinity.
		
		\item[(ii)] If $\mu \in C^{0,\alpha}(\partial\Omega)$, then $v^+_{\Omega}[\mu] \in C^{1,\alpha}(\overline{\Omega})$ and the map from $C^{0,\alpha}(\partial\Omega)$ to $C^{1,\alpha}(\overline{\Omega})$ which takes $\mu$ to $v^+_{\Omega}[\mu]$ is linear and continuous. Moreover, $v^-_{\Omega}[\mu] \in C^{1,\alpha}_{\mathrm{loc}}(\overline{\Omega^-})$ and the map from $C^{0,\alpha}(\partial \Omega)$ to $C^{1,\alpha}_{\mathrm{loc}}(\overline{\Omega^-})$ which takes $\mu$ to $v^-_{\Omega}[\mu]$ is linear and continuous.
		
		\item[(iii)] If $\mu \in C^{0,\alpha}(\partial\Omega)$, then we have following jump relations
		\begin{equation*}
			\nu_\Omega \cdot \nabla v^\pm_{\Omega}[\mu] (x) = \left( \mp \frac{1}{2} I + W^\ast_{\partial\Omega} \right)[\mu](x) \qquad \forall x \in \partial \Omega.
		\end{equation*}	
		
		\item[(iv)] The operator $V_{\partial\Omega}$ is an isomorphism from $C^{0,\alpha}(\partial\Omega)$ to $\CO$.
		
		\item[(v)] For all $\mu \in \CO$, the function $w_{\Omega}[\mu](\cdot)$ is harmonic in $\R^n\setminus \partial\Omega$ and it is harmonic at infinity. Moreover, we have the following jump relations
		\begin{equation*}
		w^\pm_{\Omega}[\mu](x) = \left( \pm \frac{1}{2} I + W_{\partial\Omega} \right)[\mu](x) \qquad \forall x \in \partial \Omega.
		\end{equation*}
		
		\item[(vi)] Let $\mu \in C^{1,\alpha}(\partial\Omega)$. Then $w^+_{\Omega}[\mu] \in C^{1,\alpha}(\overline{\Omega})$ and $w^-_{\Omega}[\mu] \in C^{1,\alpha}_{\mathrm{loc}}(\overline{\Omega^-})$ and we have
		\begin{equation*}
		\nu_\Omega \cdot \nabla w^+_{\Omega}[\mu] - \nu_\Omega \cdot \nabla w^-_{\Omega}[\mu] = 0 \qquad \text{on } \partial \Omega.
		\end{equation*}
		
		\item[(vii)] The map from $C^{1,\alpha}(\partial\Omega)$ to $C^{1,\alpha}(\overline{\Omega})$ which takes $\mu$ to $w^+_{\Omega}[\mu]$ is linear and continuous and the map from $C^{1,\alpha}(\partial\Omega)$ to $C^{1,\alpha}_{\mathrm{loc}}(\overline{\Omega^-})$ which takes $\mu$ to $w^-_{\Omega}[\mu]$ is linear and continuous.
		
		\item[(viii)]  If $u\in C^{1,\alpha}_{\mathrm{harm}}(\overline\Omega)$, then 
		\[
		{w}_\Omega[u_{|\partial\Omega}](x)-{v}_\Omega[\nu_\Omega\cdot\nabla u_{|\partial\Omega}](x)=
		\begin{cases}
		u(x)&\text{if }x\in\Omega\,,\\
		\frac{1}{2} u(x)&\text{if }x\in\partial\Omega\,,\\
		0&\text{if }x\in\Omega^-\,.
		\end{cases}
		\]
	
	\end{enumerate}
\end{teo}

Moreover, the following classical result by Schauder holds.

\begin{teo}\label{Schaudercompact}
	The map which takes $\psi$ to $W_{\partial\Omega}[\psi]$ is compact from $C^{1,\alpha}(\partial\Omega)$ to itself.
\end{teo}

We summarize some results concerning the null-spaces $\mathrm{Ker} \left(\pm\frac{1}{2} I + W_{\partial\Omega}\right)$ of the operators $\pm\frac{1}{2} I + W_{\partial\Omega}$ from $\CO$ to itself.
\begin{teo}\label{fredalt}
	Let $\Omega_1, \dots , \Omega_N$ be the bounded connected components of $\Omega$ and $\Omega^-_0, \Omega^-_1, \dots , \Omega^-_M$ be the connected components of $\Omega^-$. Assume that $\Omega^-_1, \dots , \Omega^-_M$ are bounded and that $\Omega^-_0$ is unbounded. Then the following statements hold.
	\begin{enumerate}		
		\item[(i)] $\mathrm{Ker} \left(\frac{1}{2} I + W_{\partial\Omega}\right)$ consists of the functions from $\partial \Omega$ to $\R$ which are constant on $\partial\Omega^-_j$ for all $j \in \{1, \dots , M\}$ and which are identically equal to $0$ on $\partial\Omega^-_0$.
		
		\item[(ii)] $\mathrm{Ker} \left(-\frac{1}{2} I + W_{\partial\Omega}\right)$ consists of the functions from $\partial \Omega$ to $\R$ which are constant on $\partial\Omega_j$ for all $j \in \{1, \dots , N\}$.
	\end{enumerate}
\end{teo}

Finally we recall the following important result which will be widely used in the sequel.

\begin{teo}\label{frediso}
	The following statements hold.
	\begin{enumerate}
		\item[(i)] The operators $\pm \frac{1}{2} I + W_{\partial\Omega}$ are Fredholm of index $0$ from $\CO$ to itself.
				
		\item[(ii)] If $\tau \in ]-1,1[$, then the operator $\frac{1}{2} I + \tau W^\ast_{\partial\Omega}$ is an isomorphism from $C^{0,\alpha}(\partial\Omega)$ to itself.
	\end{enumerate}
\end{teo}  
In particular, for a proof of Theorem \ref{frediso}$(ii)$, we refer to Dalla Riva and Mishuris \cite[Lem. 3.5]{DaMi15}.

\section{Preliminary results}\label{preliminaryresults}

Let $\Omega^h$ be bounded open connected subset of $\R^n$ of class $C^{1,\alpha}$ with $\R^n \setminus \overline{\Omega^h}$ connected, $0 \in \Omega^h$ and $\overline{\Omega^h}\subseteq \Omega^o$. Here the superscript ``h" stands for ``hole". In the sequel we will exploit the inequality
\begin{equation}\label{Sn<0}
\int_{\partial\Omega^h} S_n\, d\sigma<0 
\end{equation}
which follows by the fact that $S_n(x)<0$ for all $x\in\mathbb{R}^n\setminus\{0\}$. Let us define
\begin{equation*}
\Omega \equiv \Omega^o \setminus \overline{\Omega^h}.
\end{equation*}
\begin{lemma}\label{isolem}
	Let $\rho \in \R \setminus\{0\}$. The map from $\COo \times \COh_0 \times \R$ to $\COharm$ which takes $(\mu^o,\mu^h,\xi)$ to the function
\[
		u[\mu^o,\mu^h,\xi] \equiv (w^+_{\Omega^o}[\mu^o] + w^-_{\Omega^h}[\mu^h] + \rho \xi \, {S_n})_{|\overline{\Omega}} 
\]
	is an isomorphism.
\end{lemma}
\begin{proof}
	The map from $\COo \times \COh$ to $\COharm$ which takes a pair $(\phi^o,\phi^i)$ to the unique solution $u[\phi^o,\phi^h]$ of the Dirichlet problem with boundary data $\phi^o$ and $\phi^h$ on $\partial\Omega^o$ and $\partial\Omega^h$, respectively, is well known to be an isomorphism. Then we consider the operator $L=(L_1,L_2)$ from $\COo \times \COh_0 \times \R$ to $\COo \times \COh$ which takes $(\mu^o,\mu^h,\xi)$ to 
	\begin{equation*}
		\begin{split}
			& L_1[\mu^o,\mu^h,\xi] \equiv \left(\frac{1}{2}I + W_{\partial\Omega^o}\right) [\mu^o] + w^-_{\Omega^h}[\mu^h]_{|\partial \Omega^o} + \rho\, \xi\, {S_n}_{| \partial\Omega^o},
			\\
			& L_2[\mu^o,\mu^h,\xi] \equiv \left(-\frac{1}{2}I + W_{\partial\Omega^h}\right) [\mu^h] + w^+_{\Omega^o}[\mu^o]_{|\partial \Omega^h} + \rho\, \xi\, {S_n}_{| \partial\Omega^h}\,.
		\end{split}
	\end{equation*}
	We observe that we can rewrite $L$ as $L=\hat{L} + \tilde{L}$ where $\hat{L}$ and $\tilde{L}$ are the operators from $\COo \times \COh_0 \times \R$ to $\COo \times \COh$ defined by
	\begin{equation*}
		\begin{split}
			& \hat{L}[\mu^o,\mu^h,\xi] \equiv \left( \frac{1}{2} \mu^o , -\frac{1}{2} \mu^h + \rho\,  \xi \, {S_n}_{| \partial\Omega^h}\right), 
			\\
			& \tilde{L}[\mu^o,\mu^h,\xi] \equiv \left(W_{\partial\Omega^o} [\mu^o] + w^-_{\Omega^h}[\mu^h]_{|\partial \Omega^o} + \rho\, \xi \, {S_n}_{| \partial\Omega^o} , W_{\partial\Omega^h} [\mu^h] + w^+_{\Omega^o}[\mu^o]_{|\partial \Omega^h} \right).
		\end{split}
	\end{equation*}
	We observe that, for all $(\phi^o,\phi^h)\in\COo \times \COh$, we have $\hat{L}[\mu^o,\mu^h,\xi]=(\phi^o,\phi^h)$ if and only if 
	\[
	\mu^o=2\phi^o\,,\quad \xi = \frac{\int_{\partial\Omega^h} \phi^h \,d\sigma}{\rho \int_{\partial\Omega^h} {S_n} \,d\sigma }\,,\quad\mu^h = -2 f^h +2 \frac{{S_n}_{| \partial\Omega^h}}{ \int_{\partial\Omega^h} {S_n} \,d\sigma }{\int_{\partial\Omega^h} \phi^h \,d\sigma}
	\]	
	(cf.~\eqref{Sn<0}). Hence, one can exhibit a bounded inverse $\hat{L}^{(-1)}$ of $\hat{L}$ from  $\COo \times \COh$  to $\COo \times \COh_0 \times \R$ and as a consequence one deduces that $\hat{L}$ is an isomorphism. Next we observe that $\tilde{L}$ is compact.  In fact, by Theorem \ref{Schaudercompact}, the map which takes $\mu^o$ to $W_{\partial\Omega^o}[\mu^o]$ is compact from $\COo$ to itself and the map which takes $\mu^h$ to $W_{\partial\Omega^h}[\mu^h]$ is compact from $\COh_0$ to $\COh$. Moreover the map which takes $\mu^h$ to $w^-_{\Omega^h}[\mu^h]_{|\partial\Omega^o}$ is compact from $\COh_0$ to $\COo$ and the map which takes $\mu^o$ to $w^+_{\Omega^o}[\mu^o]_{|\partial \Omega^h}$ is compact from $\COo$ to $\COh$, because $\Omega^h \subset \Omega^0$ and the integrals involved display no singularities. Finally the map which takes $\xi$ to $\rho {S_n}_{| \partial\Omega^o} \xi$ is compact from $\R$ into $\COo$, because it has a finite dimensional range. So $L=\hat{L} + \tilde{L}$ is a compact perturbation of an isomorphism and henceforth a Fredholm operator of index $0$. Accordingly, in order to prove that $L$ is an isomorphism, it suffices to show that it is injective. Thus we assume that $L[\mu^o,\mu^h,\xi]=0$ and we prove that $(\mu^o,\mu^h,\xi)=(0,0,0)$. If $L[\mu^o,\mu^h,\xi]=0$, then by the jump formulas of Theorem \ref{sdp}$(v)$ and by the uniqueness of the solution of the Dirichlet problem in $\Omega$ we have
	\begin{equation}\label{isoleminjeq1}
		(w^+_{\Omega^o}[\mu^o] + w^-_{\Omega^h}[\mu^h] + \rho\,\xi\, {S_n})_{| \overline{\Omega}}  = 0.
	\end{equation}
	Hence
	\begin{equation}\label{isoleminjeq2}
		\int_{\partial\Omega^h}{\nu_{\Omega^h} \cdot \nabla (w^+_{\Omega^o}[\mu^o] + w^-_{\Omega^h}[\mu^h] + \rho\,\xi\, S_n )  \,d\sigma} = 0.
	\end{equation}
	By a standard argument based on the divergence theorem, one shows that 
	\begin{equation}
		\int_{\partial\Omega^h}{\nu_{\Omega^h} \cdot \nabla w^+_{\Omega^o}[\mu^o] \,d\sigma}=0
	\end{equation}
	and by jump relation of Theorem \ref{sdp}$(vi)$ we get
	\begin{equation}\label{isoleminjeq3}
		\int_{\partial\Omega^h}{\nu_{\Omega^h} \cdot \nabla w^-_{\Omega^h}[\mu^h] \,d\sigma} = \int_{\partial\Omega^h}{\nu_{\Omega^h} \cdot \nabla w^+_{\Omega^h}[\mu^h] \,d\sigma} = 0.
	\end{equation}
	Finally, by the definition of the double layer potential and by Theorem \ref{sdp}$(viii)$, we have
	\begin{equation}\label{isoleminjeq4}
		\int_{\partial\Omega^h}{\nu_{\Omega^h} \cdot \nabla (\rho S_n \xi )  \,d\sigma} = \rho\, \xi 	\int_{\partial\Omega^h}{\nu_{\Omega^h} \cdot \nabla S_n   \,d\sigma} = \rho\, \xi\, w^+_{\Omega^h}[1](0) = \rho\, \xi.
	\end{equation}
	Hence by \eqref{isoleminjeq2}-\eqref{isoleminjeq4}, we deduce that $\xi=0$. Then by \eqref{isoleminjeq1} we have 
	\begin{equation}\label{isoleminjeq5}
	(w^+_{\Omega^o}[\mu^o] + w^-_{\Omega^h}[\mu^h])_{| \overline{\Omega}} =0.
	\end{equation} 
	Now we consider the function $\mu \in \CO$ defined by
	\begin{equation*}
		\mu (x) \equiv
		\begin{cases}
			\mu^o(x) & \mbox{if } x \in \partial\Omega^o,
			\\
			- \mu^h(x) & \mbox{if } x \in \partial\Omega^h.
		\end{cases}
	\end{equation*}
	Equality  \eqref{isoleminjeq5}, the jump relations of Theorem \ref{sdp}$(v)$, and the fact that
	\begin{equation*}
		\nu{_{\Omega}}(x)= -\nu_{\Omega^h}(x) \qquad\forall x \in \partial\Omega^h,
	\end{equation*}
	imply that  $\left(\frac{1}{2} I + W_{\partial\Omega}\right) [\mu] = 0$. Then, by Theorem \ref{fredalt}$(i)$, we obtain that $\mu^o=0$ on $\partial\Omega^o$ and $\mu^h$ is constant on $\partial\Omega^h$. Since $\mu^h \in \COh_0$, it follows that $\mu^h=0$ and we conclude that $(\mu^o,\mu^h,\xi)=(0,0,0)$. Hence $L$ is  injective and our proof is complete.
\end{proof}

\begin{lemma}\label{isolem2}
	The map from $\COi_0 \times \R$ to $\COi$ which takes $(\mu,\xi)$ to the function
	\begin{equation*}
		J[\mu,\xi] \equiv \left(-\frac{1}{2}I + W_{\partial\Omega^i}\right) [\mu] +\xi\,{S_n}_{|\partial \Omega^i} 
	\end{equation*}
	is an isomorphism.
\end{lemma}

\begin{proof}
	 We write 
	\begin{equation*}
		J[\mu,\xi] = \left(-\frac{1}{2} \mu + \xi\, {S_n}_{|\partial \Omega^i} \right)+W_{\partial\Omega^i}[\mu].
	\end{equation*}
	Then we observe that the map which takes $(\mu,\xi)\in\COi_0 \times \R$ to $-\frac{1}{2} \mu + \xi \, {S_n}_{|\partial \Omega^i}  \in\COi$ is an isomorphism with inverse given by 
	\[
	h\mapsto \left(-2 \left(h-
\frac{\int_{\partial\Omega^i} h \,d\sigma}{\int_{\partial\Omega^i} S_n \,d\sigma} {S_n}_{|\partial \Omega^i} \right) \, , \,  \frac{\int_{\partial\Omega^i} h \,d\sigma}{\int_{\partial\Omega^i} S_n \,d\sigma}\right)
\]
(cf.~\eqref{Sn<0}).
 Moreover, $W_{\partial\Omega^i}$ is compact from $\COi_0$ to $\COi$ by Theorem \ref{Schaudercompact}. Hence, $J$ is a compact perturbation of an isomorphism and therefore a Fredholm operator of index $0$. Accordingly, to prove that it is an isomorphism it suffices to show that it is injective.  Let $(\mu,\xi) \in \COi_0 \times \R$ be such that 
	\begin{equation}\label{isolem2.eq2}
		J[\mu,\xi] = \left(-\frac{1}{2}I + W_{\partial\Omega^i}\right) [\mu] +\xi\, {S_n}_{|\partial \Omega^i}  = 0.
	\end{equation}
	Then  by Theorem \ref{sdp} $(v)$ we have that $w^-_{\Omega^i}[\mu]_{|\partial \Omega^i} = - \xi\, {S_n}_{|\partial \Omega^i}$ and, by the uniqueness of the solution of the exterior Dirichlet  problem in $\Omega^{i-}$ (note that  $S_n$ and $w^-_{\Omega^i}[\mu]$ are both harmonic at infinity), we deduce that
	\begin{equation}\label{w-=-xi Sn}
	w^-_{\Omega^i}[\mu](x) = -\xi\, {S_n} (x) \quad \forall x \in \overline{\Omega^{i-}}.
	\end{equation}
	Then we observe that 
	\begin{equation*}
		\lim\limits_{|x| \rightarrow + \infty} (n-2)s_n|x|^{n-2}{w^-_{\Omega^i}[\mu](x)}=0, \quad \lim\limits_{|x| \rightarrow + \infty} (n-2)s_n|x|^{n-2}(-\xi\, S_n (x)) = -\xi 
	\end{equation*}
	(recall that here $n\ge 3$). Hence $\xi =0$ by \eqref{w-=-xi Sn}. Then, $\left(-\frac{1}{2}I + W_{\partial\Omega^i}\right) [\mu] = 0$ by \eqref{isolem2.eq2}. Finally, by Theorem \ref{fredalt}$(ii)$, and by the membership of  $\mu$ in $\COi_0$, we also have $\mu=0$. Hence $(\mu,\xi)=(0,0)$ and the proof is completed.
\end{proof}

\section{Formulation of the problem in terms of integral equations}\label{formulation}

In this section, we  introduce a formulation of problem \eqref{princeq} in terms of integral equations.  In the sequel  we denote by ${u}^o$  the unique solution in $C^{1,\alpha}(\overline{\Omega^o})$ of the interior Dirichlet problem in $\Omega^o$ with boundary datum $f^o$, namely
\[
\begin{cases}
\Delta {u}^o=0&\text{in }\Omega^o\,,\\
{u}^o=f^o&\text{on }\partial\Omega^o\,.
\end{cases}
\]
We indicate by $\partial_\epsilon F$ and $\partial_\zeta F$ the partial derivative of $F$ with respect to the first the last argument, respectively.
We shall exploit the following assumption:
\begin{equation}\label{zetaicond}
\begin{split}
&\text{There exists $\zeta^i \in \R$ such that $F(0,\cdot,\zeta^i) = {u}^o(0)$}\\
&\text{and $(\partial_\zeta F)(0,\cdot,\zeta^i)$ is constant and positive.}
\end{split}
\end{equation}

Then we have the following Proposition \ref{rappresentsol}, where we represent harmonic functions in $\overline{\Omega(\epsilon)}$ and $\overline{\epsilon \Omega^i}$ in terms of ${u}^o$, double layer potentials with appropriate densities, and a suitable restriction of the  fundamental solution $S_n$.

\begin{prop}\label{rappresentsol}
	Let $\epsilon \in ]0,\epsilon_0[$. The map $(U^o_\epsilon[\cdot,\cdot,\cdot,\cdot], U^i_\epsilon[\cdot,\cdot,\cdot,\cdot])$ from $\COo \times \COi_0 \times \R \times \COi$ to $C^{1,\alpha}_{\mathrm{harm}}(\overline{\Omega(\epsilon)}) \times \COeiharm$ which takes $(\phi^o,\phi^i,\zeta,\psi^i)$ to the pair of functions
	\begin{equation*}
	(U^o_\epsilon[\phi^o,\phi^i,\zeta,\psi^i],U^i_\epsilon[\phi^o,\phi^i,\zeta,\psi^i])
	\end{equation*}
	defined by
	\begin{equation}\label{rappsol}
		\begin{aligned}
			U^o_\epsilon[\phi^o,\phi^i,\zeta,\psi^i](x) &\equiv {u}^o(x) + \epsilon w^+_{\Omega^o}[\phi^o](x) + \epsilon w^-_{\epsilon\Omega^i}\left[\phi^i\left(\frac{\cdot}{\epsilon}\right)\right](x) 
			\\
			& \quad+ \epsilon^{n-1}\zeta\, S_n(x)  && \qquad \forall x \in \overline{\Omega(\epsilon)}, 
			\\
			U^i_\epsilon[\phi^o,\phi^i,\zeta,\psi^i](x) &\equiv \epsilon w^+_{\epsilon\Omega^i}\left[\psi^i\left(\frac{\cdot}{\epsilon}\right)\right](x) + \zeta^i && \qquad \forall x \in \epsilon \overline{\Omega^i},
		\end{aligned}
	\end{equation}
	is bijective.
\end{prop}

\begin{proof}
Let $\epsilon \in ]0,\epsilon_0[$ and $(v^o,v^i) \in \COoiharm \times \COeiharm$. We prove that there exists a unique quadruple $(\phi^o,\phi^i,\zeta,\psi^i) \in \COo \times \COi_0 \times \R \times \COi$ such that 
\begin{equation}\label{rappsol.eq1}
(U^o_\epsilon[\phi^o,\phi^i,\zeta,\psi^i],U^i_\epsilon[\phi^o,\phi^i,\zeta,\psi^i]) = (v^o,v^i)\,.
\end{equation}
Indeed, \eqref{rappsol.eq1} is equivalent to  
\begin{align} 
\label{rappsol.eq2}
& w^+_{\Omega^o}[\phi^o](x) + w^-_{\epsilon\Omega^i}\left[\phi^i\left(\frac{\cdot}{\epsilon}\right)\right](x) + \epsilon^{n-2} \zeta \, S_n(x)  = \frac{1}{\epsilon}(v^o(x) - {u}^o(x)) && \forall x \in \overline{\Omega(\epsilon)},
\\
\label{rappsol.eq3}
& w^+_{\epsilon\Omega^i}\left[\psi^i\left(\frac{\cdot}{\epsilon}\right)\right](x) = \frac{1}{\epsilon}(v^i(x) - \zeta^i) &&\forall x \in \epsilon \overline{\Omega^i}.
\end{align}
Since $\frac{1}{\epsilon}(v^o - {u}^o) \in \COoiharm$, the existence and uniqueness of $(\phi^o,\phi^i,\xi) \in \COo \times \COi_0 \times \R$ which satisfies \eqref{rappsol.eq2} follow from Lemma \ref{isolem} (with $\rho=\epsilon^{n-2}$). By Theorem \ref{sdp}$(v)$ and by the uniqueness of the solution of the interior Dirichlet problem, equation \eqref{rappsol.eq3} is equivalent to 
$\left(\frac{1}{2}I+W_{\epsilon\partial\Omega^i}\right)\left[\psi^i\left(\frac{\cdot}{\epsilon}\right) \right]=\frac{1}{\epsilon}(v^i - \zeta^i)_{|\epsilon\partial\Omega^i}$. By Theorems \ref{fredalt} and \ref{frediso} one verifies that $\frac{1}{2}I+W_{\epsilon\partial\Omega^i}$ is an isomorphism from  $C^{1,\alpha}(\epsilon\partial\Omega^i)$ to itself. Hence there exists a unique $\psi^i \in \COi$  solution   of \eqref{rappsol.eq3}.
\end{proof}

In the following lemma we consider the Taylor expansion of the function $F$. In addition, we assume that:
\begin{equation}\label{F1}
\begin{split}
&\mbox{For all $t\in\partial\Omega^i$ fixed, the map from $]-\epsilon_0,\epsilon_0[ \times \R$ to $\R$}
\\
&\mbox{which takes  $(\epsilon,\zeta)$  to  $F(\epsilon,t,\zeta)$ is of class $C^2$.}  
\end{split}
\end{equation} 

\begin{lemma}\label{Taylem}
	Let \eqref{F1} hold true. Let $a,b \in \R$. Then
	\begin{equation*}
		F(\epsilon,t,a+\epsilon b) = F(0,t,a) + \epsilon (\partial_\epsilon F) (0,t,a) + \epsilon b (\partial_\zeta F) (0,t,a) + \epsilon^2 \tilde{F}(\epsilon,t,a,b),
	\end{equation*}
	for all $(\epsilon,t) \in ]-\epsilon_0,\epsilon_0[ \times \partial \Omega^i$,
	where 
	\begin{equation}\label{Taylem.eq1}
		\begin{split}
			\tilde{F}(\epsilon,t,a,b) \equiv \int_{0}^{1}  (1-\tau) \{ (\partial^2_\epsilon F)(\tau\epsilon,t,a + \tau\epsilon b) & +  2b(\partial_\epsilon \partial_\zeta F)(\tau\epsilon,t,a + \tau\epsilon b) 
			\\
			& + b^2(\partial^2_\zeta F)(\tau\epsilon,t,a + \tau\epsilon b) \} \,d\tau .
		\end{split}	
	\end{equation}
\end{lemma}

\begin{proof}
	It suffices to consider the following identities:
	\begin{equation*}
		d_\epsilon ( F (\epsilon,t,a+\epsilon b)) = \partial_\epsilon F (\epsilon,t,a+\epsilon b) + b \partial_\zeta F (\epsilon,t,a+\epsilon b)
	\end{equation*}
	and 
	\begin{equation}\label{Taylem.eq2}
	\begin{split}
	d^2_\epsilon ( F (\epsilon,t,a+\epsilon b) ) = (\partial^2_\epsilon F)(\epsilon,t,a + \epsilon b) &+ 2b(\partial_\epsilon \partial_\zeta F)(\epsilon,t,a + \epsilon b) \\
	&+ b^2 (\partial^2_\zeta F)(\epsilon,t,a + \epsilon b),
	\end{split}
	\end{equation}
	and to take the Taylor expansion of $F(\epsilon,t,a+\epsilon b)$ with respect to $\epsilon$ and with remainder in integral form.
\end{proof}

Moreover, under assumption \eqref{zetaicond}, we have the following technical Lemma \ref{tayuo}.

\begin{lemma}\label{tayuo}
Let \eqref{zetaicond} hold true. Then
\begin{equation}\label{tayuo.eq1}
u^o(\epsilon t)-F(0,t,\zeta^i)=\epsilon\, t\cdot\nabla u^o(0)+\epsilon^2\,\tilde{u}^o(\epsilon,t)\qquad\forall\epsilon\in]-\epsilon_0,\epsilon_0[\,,\ t\in\partial\Omega^i
\end{equation}
with 
\begin{equation}\label{tayuo.eq2}
\tilde{u}^o(\epsilon,t) \equiv \int_0^1(1-\tau)\sum_{i,j=1}^nt_i\,t_j\, (\partial_{x_i}\partial_{x_j} u^o)(\tau\epsilon t)\,d\tau\,.
\end{equation}
Moreover, the map from $]-\epsilon_0,\epsilon_0[$ to $\COi$ which takes $\epsilon$ to $\tilde{u}^o(\epsilon,\cdot)$ is real analytic.
\end{lemma}

\begin{proof} To prove \eqref{tayuo.eq1} and \eqref{tayuo.eq2} it suffices to take the Taylor expansion of $u^o(\epsilon t)$ with respect to $\epsilon$ and with remainder in integral form (see also \eqref{zetaicond}). Then we observe that the map from $]-\epsilon_0,\epsilon_0[$ to $(C^{1,\alpha}([0,1]\times\overline{\Omega^i}))^n$ which takes $\epsilon$ to the function $\epsilon\tau t$ of the variable $(\tau,t)$ is real analytic. Moreover we have $\epsilon\tau t\in\Omega^o$ for all $\epsilon\in]-\epsilon_0,\epsilon_0[$ and all $(\tau,t)\in[0,1]\times\overline{\Omega^i}$. Then, by the real analyticity of $\partial_{x_i}\partial_{x_j} u^o$ in $\Omega^o$ and by known results on composition operators (cf.~Valent \cite[Thm.~5.2, p.~44]{Va88}), one verifies that the map from
\begin{equation*}
\{h \in (C^{1,\alpha}([0,1]\times\overline{\Omega^i}))^n : h([0,1]\times\overline{\Omega^i}) \subset \Omega^o \}
\end{equation*}
to $C^{1,\alpha}([0,1]\times\overline{\Omega^i})$ which takes a function $h(\cdot,\cdot)$ to $\partial_{x_i}\partial_{x_j} u^o(h(\cdot,\cdot))$ is real analytic. Since the sum and the pointwise product in $C^{1,\alpha}([0,1]\times\overline{\Omega^i})$ are bilinear and continuous, the map from $]-\epsilon_0,\epsilon_0[$ to $C^{1,\alpha}([0,1]\times\overline{\Omega^i})$ which takes $\epsilon$ to the function
\[
(1-\tau)\sum_{i,j=1}^nt_i\,t_j\, (\partial_{x_i}\partial_{x_j} u^o)(\tau\epsilon t)
\]
of the variable $(\tau,t)$ is real analytic. Then, since the map from $C^{1,\alpha}([0,1]\times\overline{\Omega^i})$ to $C^{1,\alpha}(\overline{\Omega^i})$ which takes a function $g(\cdot,\cdot)$ to $\int_{0}^{1} g(\tau,\cdot) d\,\tau$ is linear and continuous and since the restriction operator is linear and continuous from $C^{1,\alpha}(\overline{\Omega^i})$ to $\COi$, 
 we conclude that the map from $]-\epsilon_0,\epsilon_0[$ to $\COi$ that takes $\epsilon$ to $\tilde{u}(\epsilon,\cdot)$
is real analytic. 
\end{proof}

We are now ready to provide a formulation of problem \eqref{princeq} in terms of integral equations.

\begin{prop}\label{equivsol}
Let assumptions  \eqref{zetaicond}  and \eqref{F1} hold true. Let $\epsilon \in ]0,\epsilon_0[$ and $(\phi^o,\phi^i,\zeta,\psi^i) \in \COo \times \COi_0 \times \R \times \COi$. Then the pair of functions
\[
(U^o_\epsilon[\phi^o,\phi^i,\zeta,\psi^i],U^i_\epsilon[\phi^o,\phi^i,\zeta,\psi^i])
\]
defined by \eqref{rappsol} is a solution of \eqref{princeq} if and only if 
\begin{align}
\label{equivsol3}
&\left(\frac{1}{2}I + W_{\partial\Omega^o}\right)[\phi^o](x) - \epsilon^{n-1} \int_{\partial\Omega^i}{\nu_{\Omega^i}(y) \cdot \nabla S_n(x-\epsilon y) \phi^i(y) \,d\sigma_y} \nonumber
\\& \qquad
+ \epsilon^{n-2} \zeta\, S_n(x)  = 0 && \forall x \in \partial\Omega^o\,, 
\\
\label{equivsol1}
& t \cdot \nabla {u}^o(0) + \epsilon \tilde{u}^o(\epsilon, t) + \left(-\frac{1}{2}I + W_{\partial\Omega^i}\right)[\phi^i](t) +\zeta\, S_n(t) + w^+_{\Omega^o}[\phi^o](\epsilon t) \nonumber
\\& \qquad 
= (\partial_\epsilon F) (0,t,\zeta^i) \nonumber + (\partial_\zeta F) (0,t,\zeta^i) \left(\frac{1}{2}I + W_{\partial\Omega^i}\right)[\psi^i](t) 
\\
& \qquad
+ \epsilon \tilde{F}\left(\epsilon,t,\zeta^i,\left(\frac{1}{2}I + W_{\partial\Omega^i}\right)[\psi^i](t)\right) &&\forall t \in \partial\Omega^i,
\\
\label{equivsol2}
& \nu_{\Omega^i}(t) \cdot \Big( \nabla {u}^o(\epsilon t) + \epsilon \nabla w^+_{\Omega^o}[\phi^o](\epsilon t) + \nabla w^-_{\Omega^i}[\phi^i](t) \nonumber
\\&
\qquad + \zeta\,\nabla S_n(t)  - \nabla w^+_{\Omega^i}[\psi^i](t) \Big)
\nonumber
\\
& \qquad = G \left(\epsilon,t,\epsilon\left(\frac{1}{2}I + W_{\partial\Omega^i}\right)[\psi^i](t) + \zeta^i \right)&& \forall t \in \partial\Omega^i\,.
\end{align}
\end{prop}

\begin{proof}
It can be deduced by definition \eqref{rappsol}, by the jump formulas of Theorem \ref{sdp}$(v)$, by changing the variable $x$ with $\epsilon t$ in the integral equations on $\epsilon\partial\Omega^i$ and in the integrals over $\epsilon\partial\Omega^i$,  by Lemma \ref{Taylem} with $a=\zeta^i$ and $b=\left(\frac{1}{2}I + W_{\partial\Omega^i}\right)[\psi^i](x)$, and by Lemma \ref{tayuo}.
\end{proof}

Incidentally we observe that, by integrating \eqref{equivsol2} over $\partial\Omega^i$, one shows that
\[
\zeta=\int_{\partial\Omega^i}G \left(\epsilon,t,\epsilon \left(\frac{1}{2}I + W_{\partial\Omega^i}\right)[\psi^i](t) + \zeta^i \right)\,d\sigma_t
\]
for all $\epsilon\in]0,\epsilon_0[$ (see also Theorem \ref{sdp}$(viii)$).


\section{Limiting system}\label{limitingsystem}
In this section we prove an existence and uniqueness theorem for the limiting system, i.e.~for the system of integral equations obtained by letting $\epsilon \rightarrow 0^+$ in \eqref{equivsol3}, \eqref{equivsol1} and \eqref{equivsol2}. It consists of the following three equations:

\begin{align}\label{limsys}
&\left(\frac{1}{2}I + W_{\partial\Omega^o}\right)[\phi^o](x) = 0 & \forall x \in \partial\Omega^o, \nonumber
\\
& t \cdot \nabla {u}^o(0) + \left(-\frac{1}{2}I +  W_{\partial\Omega^i}\right)[\phi^i](t)  +\zeta\, S_n(t)  +w^+_{\Omega^o}[\phi^o](0) \nonumber
\\ 
& \qquad\qquad =(\partial_\epsilon F) (0,t,\zeta^i) + (\partial_\zeta F) (0,t,\zeta^i)  \left(\frac{1}{2}I + W_{\partial\Omega^i}\right)[\psi^i](t) & \forall t \in \partial\Omega^i, \nonumber
\\
&\nu_{\Omega^i}(t) \cdot \left( \nabla {u}^o(0) + \nabla w^-_{\Omega^i}[\phi^i](t)+\zeta\, \nabla S_n(t)  - \nabla w^+_{\Omega^i}[\psi^i](t) \right) = G (0,t,\zeta^i) & \forall t \in \partial\Omega^i.
\end{align}

We begin with an existence and uniqueness result for an auxiliary exterior transmission problem.
\begin{lemma}\label{limsyslem}
Let $\lambda>0$. Let $f_1 \in \COi$ and $f_2 \in C^{0,\alpha}(\partial\Omega^i)$. Then there exists a unique solution $(u^-,u^+) \in C^{1,\alpha}(\R^n\setminus\Omega^i) \times C^{1,\alpha}(\overline{\Omega^i})$ of 
\begin{equation}\label{aux}
\begin{cases}
\Delta u^- = 0 & \mbox{in } \R^n \setminus \overline{\Omega^i}, \\
\Delta u^+ = 0 & \mbox{in } \Omega^i, \\
u^-=\lambda u^+ + f_1 & \mbox{on } \partial\Omega^i,  
\\
\nu_{\Omega^i} \cdot \nabla u^- - \nu_{\Omega^i} \cdot \nabla u^+ = f_2 & \mbox{on } \partial\Omega^i,  
\\
\lim\limits_{x \rightarrow \infty} u^-(x) = 0.
\end{cases}
\end{equation}
\end{lemma}

\begin{proof}
	We first prove the uniqueness. By linearity it suffices to show that the only solution with $f_1=f_2=0$ is $(u^-,u^+)=(0,0)$. If $f_1=f_2=0$, then by the divergence theorem and by the harmonicity at infinity of $u^-$, we compute
\begin{equation*}
\begin{split}
0 \leq \int_{\Omega^i} |\nabla u^+|^2 \,dx = - \int_{\Omega^i} u^+ \Delta u^+ \,dx +  \int_{\partial \Omega^i} u^+ \, \nu_{\Omega^i} \cdot \nabla u^+ \,d\sigma 
\\ = \int_{\partial \Omega^i} \frac{1}{\lambda}u^- \, \nu_{\Omega^i} \cdot \nabla u^- \,d\sigma = - \frac{1}{\lambda}\int_{\R^n \setminus \overline{\Omega^i}} |\nabla u^-|^2 \,dx \leq 0.
\end{split}	
\end{equation*}
It follows that $u^-$ and $u^+$ are constant functions (note that $\lambda>0$), hence the fifth condition of \eqref{aux} implies that $u^-=0$ and, in turn, the third condition (with $f_1=0$) implies that $u^+=0$. 

Now we show that the solution of \eqref{aux} exists for any $f_1 \in \COi$ and $f_2 \in C^{0,\alpha}(\partial\Omega^i)$ fixed. 	To do so,  we prove that there exists a (unique) pair $(\phi,\mu) \in C^{0,\alpha}(\partial\Omega^i) \times C^{0,\alpha}(\partial\Omega^i)$ such that the pair of functions $(u^-,u^+) \in C^{1,\alpha}_{\mathrm{harm}}(\R^n\setminus\Omega^i) \times \COiharm$ defined by 
\begin{equation*}
\begin{aligned}
& u^- \equiv \lambda v^-_{\Omega^i}[\phi]  &&\qquad \text{on } \R^n \setminus \Omega^i,
\\
& u^+ \equiv v^+_{\Omega^i}[\phi + \mu]  &&\qquad \text{on } \overline{\Omega^i},
\end{aligned}
\end{equation*}
is a solution of \eqref{aux}. Indeed, since $v^+_{\Omega^i}[\phi]=v^-_{\Omega^i}[\phi]$ on $\partial\Omega^i$, the third equation of \eqref{aux} is equivalent to 
\begin{equation*}
	V_{\partial\Omega^i}[\mu]  = \frac{1}{\lambda} f_1 \qquad \text{on } \partial \Omega^i,
\end{equation*}
and  then the existence (and uniqueness) of $\mu \in C^{0,\alpha}(\partial\Omega^i)$ follows by Theorem \ref{sdp}$(iv)$. Moreover, by the jump relations for the normal derivative of the single layer potential (see Theorem \ref{sdp}$(iii)$), we deduce that the  fourth equation of \eqref{aux} is equivalent to 
\begin{equation*}
	\lambda \left(\frac{1}{2}I + W^\ast_{\partial\Omega}\right)[\phi] - \left(-\frac{1}{2}I + W^\ast_{\partial\Omega}\right)[\phi+\mu] = f_2 \qquad \text{on } \partial \Omega^i.
\end{equation*}
By straightforward computation we obtain
\begin{equation*}
\left(\frac{1}{2}I + \frac{\lambda-1}{\lambda+1} \, W^\ast_{\partial\Omega}\right)[\phi] = \frac{1}{\lambda+1} \left( f_2 + \left(-\frac{1}{2}I + W^\ast_{\partial\Omega}\right)[\mu] \right)  \qquad \text{on } \partial \Omega^i,
\end{equation*}
and the existence (and uniqueness) of $\phi\in C^{0,\alpha}(\partial\Omega^i)$ comes from Theorem \ref{frediso}$(ii)$ (note that $\left|\frac{\lambda-1}{\lambda+1}\right| < 1$). To complete the proof, we observe that $\lim_{x\to\infty}v^-_{\Omega^i}[\phi](x)=0$ and thus $u^-$ satisfies also the last equation of \eqref{aux}. 
\end{proof}

We now get back to the analysis of \eqref{limsys}. 

\begin{teo}\label{limsysthm}
Let assumptions  \eqref{zetaicond} and \eqref{F1} hold true. Then, the quadruple $(\phi^o_0,\phi^i_0,\zeta_0,\psi^i_0) \in \COo \times \COi_0 \times \R \times \COi$ is a solution of \eqref{limsys} if and only if 
\[
\phi^o_0=0
\]
and the pair of functions $(v^-,v^+) \in C^{1,\alpha}(\R^n\setminus\Omega^i) \times C^{1,\alpha}(\overline{\Omega^i})$ defined by
\begin{equation}\label{limsysthm.eq1}
\begin{aligned}
& v^- \equiv w^-_{\Omega^i}[\phi^i_0] +   \zeta_0\,  S_n &&\qquad \text{on } \R^n \setminus \Omega^i,
\\
& v^+ \equiv w^+_{\Omega^i}[\psi^i_0]  &&\qquad \text{on } \overline{\Omega^i},
\end{aligned}
\end{equation}
is a solution of 
	\begin{equation}\label{LP}
		\begin{cases}
			\Delta v^- = 0 & \mbox{in } \R^n \setminus \overline{\Omega^i}, \\
			\Delta v^+ = 0 & \mbox{in } \Omega^i, \\
			v^-(t) = (\partial_\zeta F) (0,t,\zeta^i) v^+(t) + (\partial_\epsilon F) (0,t,\zeta^i) - t \cdot \nabla{u}^o(0) & \forall t\in \partial\Omega^i,  \\
			\nu_{\Omega^i}(t) \cdot \nabla v^-(t) - \nu_{\Omega^i}(t) \cdot \nabla v^+(t) = G(0,t,\zeta^i) - \nu_{\Omega^i}(t) \cdot \nabla{u}^o(0) & \forall t \in \partial\Omega^i, \\
			\lim\limits_{t \to\infty} v^-(t)=0.
		\end{cases} 
	\end{equation}	
In particular, there exists a  unique solution  $(v^-,v^+) \in C^{1,\alpha}(\R^n\setminus\Omega^i) \times C^{1,\alpha}(\overline{\Omega^i})$ of \eqref{LP} and a unique solution $(\phi^o_0,\phi^i_0,\zeta_0,\psi^i_0) \in \COo \times \COi_0 \times \R \times \COi$  of \eqref{limsys}.
\end{teo}

\begin{proof}
By Theorem \ref{fredalt}$(i)$ and by Theorem \ref{frediso}$(i)$ the only solution of the first equation of \eqref{limsys} is  $\phi^o_0=0$.  Then, by Theorem \ref{sdp}$(v)$, one verifies that the triple $(\phi^i_0,\zeta_0,\psi^i_0) \in  \COi_0 \times \R \times \COi$  is a solution of the last two equations of \eqref{limsys} if and only if the pair  $(v^-,v^+)$ defined by \eqref{limsysthm.eq1} is a solution of \eqref{LP}.  In addition, Lemma \ref{limsyslem} implies that   \eqref{LP}  has a unique solution $(v^-,v^+)\in C^{1,\alpha}_{\mathrm{harm}}(\R^n\setminus\Omega^i) \times \COiharm$. Then the existence and uniqueness of   $(\phi^i_0,\zeta_0) \in \COi_0 \times \R$ follow by the uniqueness of the solution of the exterior Dirichlet problem, by the jump relations of Theorem \ref{sdp}$(v)$, and by Lemma \ref{isolem2}. Finally, the existence and uniqueness of $\psi^i_0 \in \COi$ can be deduced by the uniqueness of the solution of the  Dirichlet problem, by Theorem \ref{sdp}$(v)$, by Theorem \ref{fredalt}$(i)$ and by Theorem \ref{frediso}$(i)$.
\end{proof}

We incidentally observe that by integrating the third equation of \eqref{limsys} over $\partial\Omega^i$ we get
\[
\zeta_0=\int_{\partial\Omega^i}G (0,t,\zeta^i)\,d\sigma_t
\]
(cf.~Theorem \ref{sdp}$(viii)$).


\section{Application of the implicit function theorem}\label{applicationIFT}

In  view of  the equivalence of problem \eqref{princeq} and equations \eqref{equivsol3}, \eqref{equivsol1}, and \eqref{equivsol2}, we now introduce the auxiliary map $M=(M_1,M_2,M_3)$ from $]-\epsilon_0,\epsilon_0[  \times \COo \times \COi_0 \times \R \times \COi$ to $\COo \times \COi \times C^{0,\alpha}(\partial\Omega^i)$ defined by 
\begin{align*}
M_1[\epsilon, \phi^o, & \phi^i, \zeta, \psi^i](x) \equiv  \left(\frac{1}{2}I + W_{\partial\Omega^o}\right)[\phi^o](x) 
\\ & -  \epsilon^{n-1} \int_{\partial\Omega^i}{\nu_{\Omega^i}(y) \cdot \nabla S_n(x-\epsilon y) \phi^i(y) \,d\sigma_y} +  \epsilon^{n-2}\zeta \, S_n(x) && \forall x \in \partial\Omega^o,
\\
M_2[\epsilon, \phi^o, &\phi^i, \zeta, \psi^i](t) \equiv  \, t \cdot \nabla {u}^o(0) + \epsilon \tilde{u}^o(\epsilon, t) + \left(-\frac{1}{2}I + W_{\partial\Omega^i}\right)[\phi^i](t)  
\\
& + \zeta\, S_n(t)  + w^+_{\Omega^o}[\phi^o](\epsilon t) - (\partial_\epsilon F) (0,t,\zeta^i) - (\partial_\zeta F) (0,t,\zeta^i) 
\\
& \times \left(\frac{1}{2}I + W_{\partial\Omega^i}\right)[\psi^i](t) - \epsilon \tilde{F}\left(\epsilon,t,\zeta^i,\left(\frac{1}{2}I + W_{\partial\Omega^i}\right)[\psi^i](t)\right) && \forall t \in \partial\Omega^i, 
\\
M_3[\epsilon, \phi^o, &\phi^i, \zeta, \psi^i](t) \equiv \, \nu_{\Omega^i}(t) \cdot \Big( \nabla {u}^o(\epsilon t) + \epsilon \nabla w^+_{\Omega^o}[\phi^o](\epsilon t) 
\\ 
& + \nabla w^-_{\Omega^i}[\phi^i](t) +  \nabla S_n(t) \zeta - \nabla w^+_{\Omega^i}[\psi^i](t) \Big)  
\\
& - G \left(\epsilon,t,\epsilon \left(\frac{1}{2}I + W_{\partial\Omega^i}\right)[\psi^i](t) + \zeta^i \right) && \forall t \in \partial\Omega^i,
\end{align*}
for all $(\epsilon,  \phi^o, \phi^i, \zeta, \psi^i) \in ]-\epsilon_0,\epsilon_0[ \times \R \times \COo \times \COi_0 \times \R \times \COi$. Then one readily verifies the following.

\begin{prop}\label{Me=0}
Let assumptions  \eqref{zetaicond}  and \eqref{F1} hold true. Let $\epsilon\in]0,\epsilon_0[$. Then the system consisting of equations \eqref{equivsol3}, \eqref{equivsol1}, and \eqref{equivsol2} is equivalent to 
\begin{equation}\label{Me=0.e1}
M[\epsilon,  \phi^o, \phi^i, \zeta, \psi^i]=(0,0,0)\,.
\end{equation}
\end{prop}

We now wish to apply the implicit function theorem for real analytic functions (see, for example, Deimling \cite[Thm.~15.3]{De85}) to equation \eqref{Me=0.e1} around the degenerate value $\epsilon=0$. As a first step we have to analyse the regularity of the map $M$. We begin by introducing some notation. 
\begin{defin}
If $H$ is a function from  $]-\epsilon_0,\epsilon_0[ \times \partial \Omega ^i \times \R$ to $\R$, then we denote by $\mathcal{N}_H$ the (nonlinear non-autonomous) superposition operator which takes a pair $(\epsilon,v)$ consisting of  a real number  $\epsilon\in ]-\epsilon_0,\epsilon_0[$ and of a function $v$ from $\partial\Omega^i$ to $\R$ to the function $\mathcal{N}_H(\epsilon,v)$ defined by
\[
\mathcal{N}_H(\epsilon,v) (t) \equiv H(\epsilon,t,v(t))\qquad\forall t\in\partial\Omega^i\,.
\]
\end{defin} 
Here the letter ``$\mathcal{N}$" stands for ``Nemytskii operator". 
\begin{rem}\label{differenzialeN_H}
If $H$ is a function from $]-\epsilon_0,\epsilon_0[ \times \partial \Omega ^i \times \R$ to $\R$ such that the superposition operator $\mathcal{N}_H$ is real analytic from $ ]-\epsilon_0,\epsilon_0[ \times \COi$ to $\COi$, then for every $(\epsilon,\overline{v}) \in ]-\epsilon_0,\epsilon_0[ \times \COi$ we have
\begin{equation}\label{d_vN_Hformula}
d_v \mathcal{N}_H (\epsilon, \overline{v}). \tilde{v} = \mathcal{N}_{(\partial_\zeta H)} (\epsilon,\overline{v}) \tilde{v} \qquad \forall \tilde{v} \in \COi. 
\end{equation}
The same result holds replacing the domain  and the target space of the operator $\mathcal{N}_H$ with $]-\epsilon_0,\epsilon_0[ \times C^{0,\alpha}(\partial\Omega^i)$ and $C^{0,\alpha}(\partial\Omega^i)$ respectively and using functions $\overline{v},\tilde{v} \in C^{0,\alpha}(\partial\Omega^i)$ in \eqref{d_vN_Hformula}. 
\end{rem}

The proof of Remark \ref{differenzialeN_H} is a straightforward modification of the corresponding argument of Lanza de Cristoforis \cite[Prop.~6.3]{La07}.

To prove that $M$ is real analytic we will exploit the following assumption:
\begin{equation}\label{realanalhp}
\begin{split}
&\text{For all $(\epsilon,v)\in]-\epsilon_0,\epsilon_0[\times \COi$}
\\
&\text{we have $\mathcal{N}_F(\epsilon,v)\in\COi$ and $\mathcal{N}_G(\epsilon,v)\in C^{0,\alpha}(\partial\Omega^i)$.} 
\\
& \text{Moreover, the superposition operator $\mathcal{N}_F$ is real analytic from} 
\\ & \text{$ ]-\epsilon_0,\epsilon_0[ \times \COi$ to $\COi$ and the superposition operator}\\
&\text{ $\mathcal{N}_G$ is real analytic from $ ]-\epsilon_0,\epsilon_0[ \times \COi$ to $C^{0,\alpha}(\partial\Omega^i)$. }
\end{split}
\end{equation}

Then we have the following technical Lemma \ref{Fana}.

\begin{lemma}\label{Fana}
Let assumptions  \eqref{zetaicond}, \eqref{F1}, and \eqref{realanalhp} hold true. Then, the map from $]-\epsilon_0,\epsilon_0[  \times  \COi$ to $\COi$ which takes $(\epsilon,\psi^i)$ to the function
\[
\tilde{F}\left(\epsilon,t,\zeta^i,\left(\frac{1}{2}I + W_{\partial\Omega^i}\right)[\psi^i](t)\right)\qquad\forall t\in\partial\Omega^i
\]
is real analytic (see also \eqref{Taylem.eq1}).
\end{lemma}
\begin{proof}
We plan to exploit Theorem \ref{Sf(w)A(tw)} in the Appendix. We begin by observing that, by the definition of $\tilde{F}$ in \eqref{Taylem.eq1} and by equalities \eqref{Taylem.eq2} and \eqref{d_vN_Hformula},  we have

\begin{align*}
&\tilde{F}\left(\epsilon,t,\zeta^i,\left(\frac{1}{2}I + W_{\partial\Omega^i}\right)[\psi^i](t)\right)
\\
&=
\int_{0}^{1} (1-\tau)  \left\{ (\partial^2_\epsilon F)\left(\tau\epsilon,t,\zeta^i + \tau\epsilon \left(\frac{1}{2}I + W_{\partial\Omega^i}\right)[\psi^i](t)\right) \right.
\\
&\qquad + 2\left(\frac{1}{2}I + W_{\partial\Omega^i}\right)[\psi^i](t)  
\,(\partial_\epsilon \partial_\zeta F)\left(\tau\epsilon,t,\zeta^i + \tau\epsilon \left(\frac{1}{2}I + W_{\partial\Omega^i}\right)[\psi^i](t)\right) 
\\
&\qquad \left.
+ \left(\left(\frac{1}{2}I + W_{\partial\Omega^i}\right)[\psi^i]\right)^2(t) \, (\partial^2_\zeta F)\left(\tau\epsilon,t,\zeta^i + \tau\epsilon \left(\frac{1}{2}I + W_{\partial\Omega^i}\right)[\psi^i](t)\right)  \right\} \,d\tau
\\
& = \int_{0}^{1} (1-\tau) \, d^2_\epsilon \mathcal{N}_F \left(\tau\epsilon, \tau\epsilon\left(\frac{1}{2}I + W_{\partial\Omega^i}\right)[\psi^i] +\zeta^i \right).(1,1) (t) \,d\tau 
\\
&\qquad + 2 \left(\frac{1}{2}I + W_{\partial\Omega^i}\right)[\psi^i](t)  
\\
&\qquad \times \int_{0}^{1} (1-\tau) \, d_\epsilon d_v \mathcal{N}_F \left(\tau\epsilon, \tau\epsilon\left(\frac{1}{2}I + W_{\partial\Omega^i}\right)[\psi^i] +\zeta^i \right).(1,1_{\partial\Omega^i}) (t) \,d\tau 
\\
&\qquad + \left(\left(\frac{1}{2}I + W_{\partial\Omega^i}\right)[\psi^i]\right)^2(t)  
\\
&\qquad \times \int_{0}^{1} (1-\tau) \, d^2_v \mathcal{N}_F \left(\tau\epsilon, \tau\epsilon\left(\frac{1}{2}I + W_{\partial\Omega^i}\right)[\psi^i] +\zeta^i \right).(1_{\partial\Omega^i},1_{\partial\Omega^i}) (t) \,d\tau 
\end{align*}
for all $(\epsilon,t,\psi^i) \in ]-\epsilon_0,\epsilon_0[ \times \partial\Omega^i \times \COi$ (cf.~Remark \ref{differenzialeN_H}). Here above $1_{\partial \Omega^i}$  denotes the constant function identically equal to $1$ on $\partial\Omega^i$. 

Now let $A$ be the map from $]-\epsilon_0,\epsilon_0[ \times \COi$ to $\COi$ which takes a pair $h=(h_1,h_2)$ to
\begin{equation*}
A(h)(t) \equiv d^2_\epsilon \mathcal{N}_F \left(h_1, h_2 +\zeta^i \right).(1,1) (t) \qquad \forall t \in \partial \Omega^i\,.
\end{equation*}
By assumption \eqref{realanalhp} one deduces that $A$ is real analytic and thus  Theorem \ref{Sf(w)A(tw)} in the Appendix implies that the map from $]-\epsilon_0,\epsilon_0[ \times \COi$ to $\COi$ which takes $h$ to 
\[
\int_{0}^{1} (1-\tau) A(\tau h) \,d\tau
\]
is also real analytic. Then we set
\begin{equation*}
h[\epsilon,\psi^i] = (h_1[\epsilon,\psi^i],h_2[\epsilon,\psi^i]) \equiv \left( \epsilon, \epsilon \left(\frac{1}{2}I + W_{\partial\Omega^i}\right)[\psi^i] \right) 
\end{equation*}
for all $(\epsilon,\psi^i) \in ]-\epsilon_0,\epsilon_0[  \times \COi$ and we observe that the map from the space $]-\epsilon_0,\epsilon_0[  \times \COi$ to itself which takes $(\epsilon,\psi^i)$ to $h[\epsilon,\psi^i]$ is  real analytic (because the first component is linear and continuous and the second one is bilinear and continuous). Since the composition of real analytic maps is real analytic, it follows that the map  from $]-\epsilon_0,\epsilon_0[  \times \COi$ to $\COi$ which takes $(\epsilon,\psi^i)$ to 
\begin{equation*}
 \int_{0}^{1} (1-\tau) \, d^2_\epsilon \mathcal{N}_F \left(\tau\epsilon, \tau\epsilon\left(\frac{1}{2}I + W_{\partial\Omega^i}\right)[\psi^i] +\zeta^i \right).(1,1) (t) \,d\tau 
\end{equation*}
is real analytic. In a similar way, one can prove that the map from the space $]-\epsilon_0,\epsilon_0[  \times \COi$ to $\COi$ which takes $(\epsilon,\psi^i)$ to 
\begin{equation*}
\int_{0}^{1} (1-\tau) \, d_\epsilon d_v \mathcal{N}_F \left(\tau\epsilon, \tau\epsilon\left(\frac{1}{2}I + W_{\partial\Omega^i}\right)[\psi^i] +\zeta^i \right).(1,1_{\partial\Omega^i}) (t) \,d\tau
\end{equation*}
and the map from $]-\epsilon_0,\epsilon_0[  \times \COi$ to $\COi$ which takes $(\epsilon,\psi^i)$ to 
\begin{equation*}
\int_{0}^{1} (1-\tau) \, d^2_v \mathcal{N}_F \left(\tau\epsilon, \tau\epsilon\left(\frac{1}{2}I + W_{\partial\Omega^i}\right)[\psi^i] +\zeta^i \right).(1_{\partial\Omega^i},1_{\partial\Omega^i}) (t) \,d\tau 
\end{equation*}
are real analytic. The map from $\COi$ to itself which takes $\psi^i$ to the function $\left(\frac{1}{2}I + W_{\partial\Omega^i}\right)[\psi^i]$ is linear and continuous, hence real analytic. Since the product of real analytic maps is real analytic, the map from $\COi$ to itself which takes $\psi^i$ to the function $\left(\left(\frac{1}{2}I + W_{\partial\Omega^i}\right)[\psi^i]\right)^2$ is real analytic. Finally,  since the sum of real analytic maps is real analytic, we conclude that the map from $]-\epsilon_0,\epsilon_0[  \times \COi$ to $\COi$ which takes $(\epsilon,\psi^i)$ to the function
\begin{equation*}
\tilde{F}\left(\epsilon,t,\zeta^i,\left(\frac{1}{2}I + W_{\partial\Omega^i}\right)[\psi^i](t)\right) \qquad \forall t \in \partial \Omega^i
\end{equation*}
is real analytic. The lemma is now proved.
\end{proof}

We now show that $M$ is real analytic.

\begin{prop}\label{Mana}
Let assumptions  \eqref{zetaicond}, \eqref{F1}, and \eqref{realanalhp} hold true. Then, the map $M$ is real analytic  from $]-\epsilon_0,\epsilon_0[  \times \COo \times \COi_0 \times \R \times \COi$ to $\COo \times \COi \times C^{0,\alpha}(\partial\Omega^i)$. 
\end{prop}
\begin{proof}
We first show that $M_1$ is real analytic from the space $]-\epsilon_0,\epsilon_0[  \times \COo \times \COi_0 \times \R \times \COi$ to $\COo$. To do so, we analyse $M_1$ term by term. 
The map from $\COo$ to $\COo$ which takes $\phi^o$ to $\left(\frac{1}{2}I + W_{\partial\Omega^o}\right)[\phi^o]$ is linear and continuous, so real analytic. The second term can be treated in this way: one considers the integral operator from $]-\epsilon',\epsilon'[ \times L^1(\partial\Omega^i)$ to $\COo$ which takes the pair $(\epsilon,f)$ to $\int_{\partial\Omega^i}{\nu_{\Omega^i}(y) \cdot \nabla S_n(\cdot -\epsilon y) f(y) \,d\sigma_y}$. By the real analyticity of $S_n$ on $\R^n \setminus \{0\}$, by the fact that the integral kernel does not display singularities and since $\COi_0$ is linearly and continuously imbedded in $L^1(\partial\Omega^i)$, we conclude that the map from $]-\epsilon',\epsilon'[ \times \COi_0$ to $\COo$ which takes the pair $(\epsilon,\phi^i)$ to $\int_{\partial\Omega^i}{\nu_{\Omega^i}(y) \cdot \nabla S_n(\cdot -\epsilon y) \phi^i(y) \,d\sigma_y}$ is real analytic (see also Lanza de Cristoforis and Musolino \cite{LaMu13}). Finally, one easily verifies that the map from $]-\epsilon',\epsilon'[ \times \R$ to $\COo$ which takes $(\epsilon,\zeta)$ to $\epsilon^{n-2} S_n(\cdot) \zeta$ is real analytic. 

We now analyse $M_2$.
For the first term there is nothing to say, because it does not depend on $(\epsilon,  \phi^o, \phi^i, \zeta, \psi^i)$. For the second term, we invoke Lemma \ref{tayuo}. The map from $\COi_0$ to $\COi$ which takes $\phi^i$ to $\left(-\frac{1}{2}I + W_{\partial\Omega^i}\right)[\phi^i]$ is linear and continuous, so real analytic. Since continuous linear maps are real analytic, the map from $\R$ to $\COi$ which takes $\zeta$ to $\zeta \,S_n(\cdot)$ is real analytic. The map from $]-\epsilon_0,\epsilon_0[ \times \COo$ to $\COi$ which takes $\epsilon$ to $w^+_{\Omega^o}[\phi^o](\epsilon \cdot)$ can be proven to be real analytic by the properties of integral operators with real analytic kernels (see Lanza de Cristoforis and Musolino \cite{LaMu13}). For the sixth term there is nothing to say, because it does not depend on $(\epsilon,  \phi^o, \phi^i, \zeta, \psi^i)$. For the seventh term, the map from $\COi$ to $\COi$ which takes $\psi^i$ to $(\partial_\zeta F) (0,\cdot,\zeta^i)  \left(\frac{1}{2}I + W_{\partial\Omega^i}\right)[\psi^i]$ is linear and continuous and hence real analytic. Finally, for the eighth term, we invoke Lemma \ref{Fana}.
	
Then we pass to consider $M_3$. 
The map from $]-\epsilon_0,\epsilon_0[$ to $(C^{0,\alpha}(\overline{\Omega^i}))^n$ which takes $\epsilon$ to the function $\epsilon t$ of the variable $t$ is real analytic. Moreover we have $\epsilon t\in\Omega^o$ for all $\epsilon\in]-\epsilon_0,\epsilon_0[$ and all $t \in\overline{\Omega^i}$. Then, by the real analyticity of $\nu_{\Omega^i} \cdot \nabla u^o$ in $\Omega^o$ and by known results on composition operators (cf.~Valent \cite[Thm.~5.2, p.~44]{Va88}), one verifies that the map from
\begin{equation*}
\{h \in (C^{0,\alpha}(\overline{\Omega^i}))^n : h(\overline{\Omega^i}) \subset \Omega^o \}
\end{equation*}
to $C^{0,\alpha}(\overline{\Omega^i})$ which takes a function $h$ to $\nu_{\Omega^i} \cdot \nabla  u^o(h(\cdot))$ is real analytic. Since the restriction operator is linear and continuous from $C^{0,\alpha}(\overline{\Omega^i})$ to $C^{0,\alpha}(\partial\Omega^i)$, we conclude that the map from $]-\epsilon',\epsilon'[$ to $C^{0,\alpha}(\partial\Omega^i)$ which takes $\epsilon$ to $\nu_{\Omega^i} \cdot \nabla {u}^o(\epsilon \cdot)$ is real analytic.
Since continuous linear maps are real analytic, the map from $\R$ to $C^{0,\alpha}(\partial\Omega^i)$ which takes $\zeta$ to $\nu_{\Omega^i} \cdot \nabla S_n \zeta$ is real analytic. By the properties of integral operators with real analytic kernels (see Lanza de Cristoforis and Musolino \cite{LaMu13}) it follows that the map from $]-\epsilon',\epsilon'[ \times \COo$ to $C^{0,\alpha}(\partial\Omega^i)$ which takes $(\epsilon, \phi^o)$ to $\nu_{\Omega^i} \cdot \epsilon \nabla w^+_{\Omega^o}[\phi^o](\epsilon \cdot)$ is real analytic.
Since linear and continuous map are real analytic, the map from $\COi$ to $C^{0,\alpha}(\partial\Omega^i)$ which takes $\phi^i$ to $\nu_{\Omega^i} \cdot \nabla w^-_{\Omega^i}[\phi^i]$, the map from $\COi$ to $C^{0,\alpha}(\partial\Omega^i)$ which takes $\psi^i$ to $\nu_{\Omega^i} \cdot \nabla w^+_{\Omega^i}[\psi^i]$, and the map from $\COi$ to $C^{0,\alpha}(\partial\Omega^i)$ which takes $\psi^i$ to $\left(\frac{1}{2}I + W_{\partial\Omega^i}\right)[\psi^i]$ are real analytic. Since product of real analytic functions is real analytic, the map from $]-\epsilon',\epsilon[ \times \COi$ to $C^{0,\alpha}(\partial\Omega^i)$ which takes $(\epsilon,\psi^i)$ to $\epsilon\left(\frac{1}{2}I + W_{\partial\Omega^i}\right)[\psi^i] + \zeta^i$ is real analytic. Finally using hypothesis \eqref{realanalhp} and again the fact that the composition of real analytic functions is real analytic, we conclude that the map from $]-\epsilon',\epsilon[ \times \COi$ to $C^{0,\alpha}(\partial\Omega^i)$ which takes $(\epsilon,\psi^i)$ to
\[
G\left(\epsilon,\cdot,\epsilon\left(\frac{1}{2}I + W_{\partial\Omega^i}\right)[\psi^i](\cdot) + \zeta^i\right) = \mathcal{N}_G \left(\epsilon, \epsilon\left(\frac{1}{2}I + W_{\partial\Omega^i}\right)[\psi^i] +\zeta^i \right)
\]
is real analytic. 

The proof of the proposition is now complete.
\end{proof}

In order to analyse problem  \eqref{princeq} for $\epsilon>0$ close to $0$, and thus equation \eqref{Me=0.e1} for $\epsilon>0$ close to $0$, we need to consider \eqref{Me=0.e1} at the singular value $\epsilon=0$. Then, by the definition of $M$, by a straightforward computation, and by Theorem \ref{limsysthm}, we deduce the following.

\begin{prop}\label{M0=0}
Let assumptions  \eqref{zetaicond}, \eqref{F1}, and \eqref{realanalhp} hold true. Then,   equation 
\[
M[0,\phi^o, \phi^i, \zeta, \psi^i] = (0,0,0)
\]
is equivalent to the limiting system \eqref{limsys} and has one and only one solution
\[
(\phi^o_0, \phi^i_0, \zeta_0, \psi^i_0) \in \COo \times \COi_0 \times \R \times \COi.
\]
\end{prop}

Finally, we have the following Lemma \ref{dMiso} concerning the partial differential of $M$ with respect to $(\phi^o, \phi^i, \zeta, \psi^i)$ evaluated at $(0,\phi^o_0, \phi^i_0, \zeta_0, \psi^i_0)$.

\begin{lemma}\label{dMiso}
Let assumptions  \eqref{zetaicond}, \eqref{F1}, and \eqref{realanalhp} hold true. Then, the partial differential of $M$ with respect to $(\phi^o, \phi^i, \zeta, \psi^i)$ evaluated at $(0,\phi^o_0, \phi^i_0, \zeta_0, \psi^i_0)$, which we denote by
\begin{equation}\label{dMiso.eq1}
\partial_{(\phi^o,\phi^i,\zeta,\psi^i)} M[0,\phi^o_0, \phi^i_0, \zeta_0, \psi^i_0]\,,
\end{equation}
is an isomorphism from $\COo \times \COi_0 \times \R \times \COi$ to $\COo \times \COi \times C^{0,\alpha}(\partial\Omega^i)$.
\end{lemma}
\begin{proof}
By standard calculus in Banach spaces one verifies that the partial differential \eqref{dMiso.eq1} is the linear and continuous operator delivered by
\begin{align*}
&\partial_{(\phi^o,\phi^i,\zeta,\psi^i)} M_1[0,\phi^o_0, \phi^i_0, \zeta_0, \psi^i_0].(\tilde{\phi^o},\tilde{\phi^i},\tilde{\zeta},\tilde{\psi^i}) (x) 
\\
& \qquad = \left(\frac{1}{2}I + W_{\partial\Omega^o}\right)[\tilde{\phi^o}](x) & \forall x \in \partial \Omega^o,
\\
& 
\partial_{(\phi^o,\phi^i,\zeta,\psi^i)} M_2[0,\phi^o_0, \phi^i_0, \zeta_0, \psi^i_0] . (\tilde{\phi^o},\tilde{\phi^i},\tilde{\zeta},\tilde{\psi^i}) (t)
\\
&
 \qquad= \left(-\frac{1}{2}I + W_{\partial\Omega^i}\right)[\tilde{\phi^i}](t) + \tilde\zeta\, S_n(t) + w^+_{\Omega^o}[\tilde{\phi^o}](0) 
\\
&
\qquad \quad- (\partial_\zeta F) (0,t,\zeta^i)  \left(\frac{1}{2}I + W_{\partial\Omega^i}\right)[\tilde{\psi^i}](t) & \forall t \in \partial\Omega^i\,, 
\\
& \partial_{(\phi^o,\phi^i,\zeta,\psi^i)} M_3[0,\phi^o_0, \phi^i_0, \zeta_0, \psi^i_0]. (\tilde{\phi^o},\tilde{\phi^i},\tilde{\zeta},\tilde{\psi^i}) (t) 
\\
&\qquad =  \nu_{\Omega^i}(t) \left( \nabla w^-_{\Omega^i}[\tilde{\phi^i}](t) +   \tilde\zeta\, \nabla S_n(t) 
 - \nabla w^+_{\Omega^i}[\tilde{\psi^i}](t) \right) & \forall t \in \partial\Omega^i,
\end{align*}
for all $(\tilde{\phi^o},\tilde{\phi^i},\tilde{\zeta},\tilde{\psi^i}) \in \COo \times \COi_0 \times \R \times \COi$. Then, to prove that $\partial_{(\phi^o,\phi^i,\zeta,\psi^i)} M[0,\phi^o_0, \phi^i_0, \zeta_0, \psi^i_0]$ is an isomorphism of Banach spaces it will suffice to prove that it is a bijection and then apply the open mapping theorem.  So let $(g^i,h^i,h^o) \in \COi \times \COi \times \COo$. We have to prove that there exists a unique quadruple $(\bar{\phi^o},\bar{\phi^i},\bar{\zeta},\bar{\psi^i}) \in \COo \times \COi_0 \times \R \times \COi$ such that
\begin{equation}\label{diffMhomo}
\partial_{(\phi^o,\phi^i,\zeta,\psi^i)} M[0,\phi^o_0, \phi^i_0, \zeta_0, \psi^i_0] .(\bar{\phi^o},\bar{\phi^i},\bar{\zeta},\bar{\psi^i}) = (g^i,h^i,h^o).
\end{equation}
The last two equations of \eqref{diffMhomo} written in full are
\begin{equation}\label{diffMhomo2comp}
\begin{aligned}
&\left(-\frac{1}{2}I + W_{\partial\Omega^i}\right)[\bar{\phi^i}](t) +\bar\zeta\, S_n(t) + w^+_{\Omega^o}[\bar{\phi^o}](0) 
\\
& \qquad\qquad\,- (\partial_\zeta F) (0,t,\zeta^i)  \left(\frac{1}{2}I + W_{\partial\Omega^i}\right)[\bar{\psi^i}](t) = g^i (t)\,,
\\
&\nu_{\Omega^i}(t) \cdot \left( \nabla w^-_{\Omega^i}[\bar{\phi^i}](t) + \bar\zeta\, \nabla S_n(t) - \nabla w^+_{\Omega^i}[\bar{\psi^i}](t) \right) = h^i(t)\,,
\end{aligned}
\end{equation}	
for all $t\in\partial\Omega^i$. Then, by Theorem \ref{sdp}$(v)$, one verifies that the triple $(\bar{\phi^i},\bar{\zeta},\bar{\psi^i}) \in  \COi_0 \times \R \times \COi$  is a solution of system \eqref{diffMhomo2comp} if and only if the pair  $(u^-,u^+)$ defined by
\begin{equation}\label{diffMhomosol}
\begin{aligned}
& u^- \equiv w^-_{\Omega^i}[\bar{\phi^i}] + \bar\zeta\, S_{n|\R^n\setminus\Omega^i} &&\quad \mbox{in }\R^n \setminus \Omega^i,
\\
& u^+ \equiv w^+_{\Omega^i}[\bar{\psi^i}] &&\quad \mbox{in } \overline{\Omega^i},
\end{aligned}
\end{equation}
is a solution of the transmission problem
\begin{equation}\label{diffMhomoexuni}
\begin{cases}
\Delta u^-=0 &\qquad \text{in } \R^n \setminus\Omega^i, 
\\
\Delta u^+=0 &\qquad \text{in } \Omega^i, 
\\
u^- =  (\partial_\zeta F) (0,\cdot \,,\zeta^i)\; u^+ - w^+_{\Omega^o}[\tilde{\phi^o}](0) + g^i &\qquad \text{on } \partial \Omega^i,
\\
\nu_{\Omega^i} \cdot \nabla u^- - \nu_{\Omega^i} \cdot \nabla u^+ = h^i & \qquad \text{on } \partial \Omega^i,
\\
\lim\limits_{x \rightarrow \infty} u^-(x) = 0\,.
\end{cases}
\end{equation}
By assumption \eqref{zetaicond} and by Lemma \ref{limsyslem}, the solution $(u^-,u^+)$ of problem \eqref{diffMhomoexuni} exists and is unique.  Then the existence and uniqueness of   $(\bar{\phi^i},\bar{\zeta}) \in \COi_0 \times \R$ follow by the first equation of \eqref{diffMhomosol}, by the uniqueness of the solution of the exterior Dirichlet problem, by the jump relations of Theorem \ref{sdp}$(v)$, and by Lemma \ref{isolem2}. The existence and uniqueness of $\bar{\psi^i} \in \COi$ can be deduced by the second equation of \eqref{diffMhomosol}, by the  uniqueness of the solution of the  Dirichlet problem, by Theorem \ref{sdp}$(v)$, by Theorem \ref{fredalt}$(i)$ and by Theorem \ref{frediso}$(i)$. Finally, to prove that $\bar{\phi^o} $ exists and is unique we observe that the first equation of \eqref{diffMhomo} is 
\[
\left(\frac{1}{2}I + W_{\partial\Omega^o}\right)[\bar{\phi^o}] = h^o
\]
and by Theorem \ref{fredalt}$(i)$ and by Theorem \ref{frediso}$(i)$ the operator $\frac{1}{2}I + W_{\partial\Omega^o}$ is invertible from $\COo$ into itself. 
\end{proof}

We are now ready to show that there is a real analytic family of solutions of equation \eqref{Me=0.e1}. 

\begin{teo}\label{ThmM}
Let assumptions  \eqref{zetaicond}, \eqref{F1}, and \eqref{realanalhp} hold true. Then there exist $\epsilon' \in ]0,\epsilon_0[$, an open neighborhood $U_0$ of $(\phi^o_0, \phi^i_0, \zeta_0, \psi^i_0)$ in $\COo \times \COi_0 \times \R \times \COi$, and a real analytic map 
\begin{equation*}
(\Phi^o[\cdot],\Phi^i[\cdot],Z[\cdot],\Psi^i[\cdot]):\;]-\epsilon',\epsilon'[\to U_0
\end{equation*}
such that  the set of zeros of $M$ in $]-\epsilon',\epsilon'[ \times U_0$ coincides with the graph of the function $(\Phi^o[\cdot],\Phi^i[\cdot],Z[\cdot],\Psi^i[\cdot] )$. In particular,
\begin{equation}\label{ThmM.eq1}
(\Phi^o[0],\Phi^i[0],Z[0],\Psi^i[0]) = (\phi^o_0, \phi^i_0, \zeta_0, \psi^i_0).
\end{equation}
\end{teo} 

\begin{proof}
It follows by  Proposition \ref{Mana}, by Lemma \ref{dMiso}, and by the implicit function theorem for real analytic functions (see, for example, Deimling \cite[Thm.~15.3]{De85}). The validity of \eqref{ThmM.eq1} is a consequence of Proposition \ref{M0=0}.
\end{proof}


\section{Real analytic representation of the family of solution}\label{analrapprfamsol}

We are now ready to exhibit  a family of solutions of problem \eqref{princeq} for $\epsilon$ sufficiently small and describe its asymptotic behaviour in terms of real analytic functions of $\epsilon$. 

\begin{defin}\label{defstandardsolution}
Let assumptions  \eqref{zetaicond}, \eqref{F1}, and \eqref{realanalhp} hold true. Let $\epsilon'$ and $(\Phi^o[\cdot],\Phi^i[\cdot],Z[\cdot],\Psi^i[\cdot])$ be as in Theorem \ref{ThmM}. 
Then, for all $\epsilon\in]0,\epsilon'[$ we set	
\[
\begin{aligned}
& u^o_\epsilon(x) \equiv U^o_\epsilon[\Phi^o[\epsilon] , \Phi^i[\epsilon], Z[\epsilon] , \Psi^i[\epsilon]](x)&&\forall x \in \overline{\Omega(\epsilon)}\,,\\
& u^i_\epsilon(x) \equiv U^i_\epsilon[\Phi^o[\epsilon] , \Phi^i[\epsilon], Z[\epsilon], \Psi^i[\epsilon]](x)	&&\forall x \in \epsilon\overline{\Omega^i}\,,
\end{aligned}
\]
with $U^o_\epsilon[\cdot,\cdot,\cdot,\cdot]$ and $U^i_\epsilon[\cdot,\cdot,\cdot,\cdot]$ defined as in \eqref{rappsol}. 
\end{defin}

As a consequence of Propositions \ref{equivsol} and \ref{Me=0} and of Theorem \ref{ThmM} we have the following.

\begin{teo}\label{uesol}
Under assumptions  \eqref{zetaicond}, \eqref{F1}, and \eqref{realanalhp}, the pair of functions
\[
(u^o_\epsilon,u^i_\epsilon) \in C^{1,\alpha}(\overline{\Omega(\epsilon)}) \times C^{1,\alpha}(\overline{\epsilon\Omega^i})
\]
is a solution of \eqref{princeq} for all $\epsilon\in]0,\epsilon'[$.
\end{teo}

We now verify that the map which takes $\epsilon$ to (suitable restrictions of) the pair of functions $(u^o_\epsilon,u^i_\epsilon)$ admits a real analytic continuation in a neighborhood of $\epsilon=0$.

\begin{teo}\label{uanal}
Let assumptions  \eqref{zetaicond}, \eqref{F1}, and \eqref{realanalhp} hold true. Then the following statements hold.
\begin{enumerate}
\item[(i)] There exists a real analytic map 
\[
U^i_m:\;]-\epsilon',\epsilon'[\to C^{1,\alpha}(\overline{\Omega^i})
\]
such that
\[
u^i_\epsilon(\epsilon t) =  \zeta^i + \epsilon U^i_m[\epsilon](t) \qquad \forall t \in \overline{\Omega^i}
\]
for all $\epsilon \in ]0,\epsilon'[$.

\item[(ii)] Let $\Omega_M$ be a bounded open subset of $\Omega^o \setminus \{0\}$ such that $0 \notin \overline{\Omega_M}$. Let $\epsilon_M \in ]0,\epsilon'[$ be such that 
\[
\overline{\Omega_M} \cap \epsilon\overline{\Omega^i}=\emptyset \qquad \forall\epsilon \in ]-\epsilon_M,\epsilon_M[\,.
\]
Then there exists  a real analytic map 
\[
U^o_M:\;]-\epsilon_M,\epsilon_M[ \to C^{1,\alpha}(\overline{\Omega_M})
\] 
such that 
\[
u^o_\epsilon(x) = {u}^o(x) + \epsilon U^o_M[\epsilon ](x) \qquad \forall x \in \overline{\Omega_M}
\]
for all $\epsilon \in ]0,\epsilon_M[$.

\item[(iii)] Let $\Omega_m$ be a bounded open subset of $\R^n\setminus\overline{\Omega^i}$. Let $\epsilon_m \in ]0,\epsilon'[$ be such that 
\begin{equation*}
\epsilon \overline{\Omega_m} \subseteq \Omega^o \qquad \forall \epsilon \in ]-\epsilon_m,\epsilon_m[.
\end{equation*}
Then there exists a real analytic map
\[
U^o_m :\;]-\epsilon_m,\epsilon_m[\to C^{1,\alpha}(\overline{\Omega_m})
\]
such that 
\[
u^o_\epsilon(\epsilon t) = u^o(0) + \epsilon U^o_m[\epsilon](t) \qquad \forall t \in \overline{\Omega_m}
\]
for all $\epsilon \in ]0,\epsilon_m[$.
\end{enumerate}
\end{teo}
\begin{proof}
We first prove $(i)$. By \eqref{rappsol} and by Definition \ref{defstandardsolution} we have 
\begin{equation*}
u^i_\epsilon(x) = \epsilon w^+_{\epsilon\Omega^i}\left[\Psi^i[\epsilon]\left(\frac{\cdot}{\epsilon}\right)\right](x) + \zeta^i \qquad \forall x \in \epsilon \overline{\Omega^i},
\end{equation*}
for all $\epsilon \in ]0,\epsilon'[$.
Then, by a computation based on the theorem of change of variable in integrals and on the homogeneity of $\nabla S_n$ we obtain that
\begin{equation*}
\begin{split}
u^i_\epsilon(\epsilon t) & = \epsilon w^+_{\epsilon\Omega^i}\left[\Psi^i[\epsilon]\left(\frac{\cdot}{\epsilon}\right)\right](\epsilon t) + \zeta^i 
\\ & =
- \epsilon \, \epsilon^{n-1} \int_{\partial\Omega^i} \nu_{\Omega^i}(s) \cdot \nabla S_n(\epsilon t - \epsilon s) \Psi^i[\epsilon](s) \,d\sigma_s + \zeta^i 
\\ & = -\epsilon \int_{\partial\Omega^i} \nu_{\Omega^i}(s) \cdot \nabla S_n(t - s) \Psi^i[\epsilon](s) \,d\sigma_s + \zeta^i 
\\ & = \epsilon w^+_{\Omega^i}[\Psi^i[\epsilon]](t) + \zeta^i 
\qquad \forall t \in \Omega^i,
\end{split}
\end{equation*}
for all $\epsilon \in ]0,\epsilon'[$.
Then it is natural to take 
\[
U^i_m[\epsilon] \equiv w^+_{\Omega^i}[\Psi^i[\epsilon]] \qquad \forall \epsilon \in ]-\epsilon',\epsilon'[\,.
\]
Since $w^+_{\Omega^i}[\cdot]$ is linear and continuous from $\COi$ to $C^{1,\alpha}(\overline{\Omega^i})$ (cf.~Theorem \ref{sdp}$(vii)$) and $\Psi^i[\cdot]$ is real analytic (cf.~Theorem \ref{ThmM}),  we conclude that the map $U^i_m$  is real analytic. The validity of (i) is proved.
	
We now proceed with $(ii)$. By \eqref{rappsol} and by Definition \ref{defstandardsolution} we have	
\begin{align*}
 u^o_\epsilon(x) = {u}^o(x) + \epsilon w^+_{\Omega^o}[\Phi^o[\epsilon] ](x) + \epsilon w^-_{\epsilon\Omega^i}\left[\Phi^i[\epsilon]\left(\frac{\cdot}{\epsilon}\right)\right](x)
 \\
 + \epsilon^{n-1} Z[\epsilon] \,  S_n(x) \quad \forall x\in\overline{\Omega(\epsilon)}
\end{align*}
for all $\epsilon\in]0,\epsilon'[$.
Then, by changing the variable of integration over $\epsilon\partial\Omega^i$ we obtain
\begin{align*}
 u^o_\epsilon(x) = {u}^o(x) + \epsilon w^+_{\Omega^o}[\Phi^o[\epsilon] ](x)- \epsilon \int_{\partial\Omega^i} \nu_{\Omega^i}(s) \cdot \nabla S_n(x - \epsilon s) \Phi^i[\epsilon](s) \, \epsilon^{n-1}\,d\sigma_s 
 \\
 + \epsilon^{n-1} Z[\epsilon]\, S_n(x) \quad \forall x \in \overline{\Omega(\epsilon)}
\end{align*}
for all $\epsilon\in]0,\epsilon'[$. Then it is natural to define
\begin{align*}
U^o_M [\epsilon](x) \equiv  w^+_{\Omega^o}[\Phi^o[\epsilon] ](x) - \epsilon^{n-1} \int_{\partial\Omega^i} \nu_{\Omega^i}(s) \cdot \nabla S_n(x - \epsilon s) \Phi^i[\epsilon](s) \,d\sigma_s  
\\
+ \epsilon^{n-2} Z[\epsilon] \, S_n(x)  \quad \forall x \in \overline{\Omega_M}
\end{align*}
for all $\epsilon\in]-\epsilon_M,\epsilon_M[$. 

Since $\Phi^o[\cdot]$ is real analytic (cf.~Theorem \ref{ThmM}), since $w^+_{\Omega^o}[\cdot]$ is linear and continuous from $\COo$ to $C^{1,\alpha}(\overline{\Omega^o})$ (cf.~Theorem \ref{sdp}$(vii)$), and since the restriction operator from $C^{1,\alpha}(\overline{\Omega^o})$ to $C^{1,\alpha}(\overline{\Omega_M})$ is linear and continuous, then the map from $]-\epsilon_M,\epsilon_M[$ to $C^{1,\alpha}(\overline{\Omega_M})$ which takes $\epsilon$ to $w^+_{\Omega^o}[\Phi^o[\epsilon]]$ is real analytic. 
Then, one considers the operator from $]-\epsilon_M,\epsilon_M[ \times L^1(\partial \Omega^i)$ to $C^{1,\alpha}(\overline{\Omega_M})$ which takes the pair $(\epsilon,f)$ to $\int_{\partial\Omega^i} \nu_{\Omega^i}(s) \cdot \nabla S_n(\cdot - \epsilon s) f(s) \,d\sigma_s$.
By the real analyticity of $S_n$ on $\R^n\setminus\{0\}$, by the fact that the integral does not display singularities (by hypothesis $\overline{\Omega_M} \cap \epsilon\overline{\Omega^i}=\emptyset$ for all $\epsilon \in ]-\epsilon_M,\epsilon_M[$), by the real analyticity of the map from $]-\epsilon_M,\epsilon_M[$ to $\COi_0$ which takes $\epsilon$ to $\Phi^i[\epsilon]$ (cf.~Theorem \ref{ThmM}) and since $\COi_0$ is linearly and continuously imbedded in $L^1(\partial \Omega^i)$, we conclude that the map from $]-\epsilon_M,\epsilon_M[$ to $C^{1,\alpha}(\overline{\Omega_M})$ which takes $\epsilon$ to $\epsilon^{n-1} \int_{\partial\Omega^i} \nu_{\Omega^i}(s) \cdot \nabla S_n(\cdot - \epsilon s) \Phi^i[\epsilon](s) \,d\sigma_s $ is real analytic (see Lanza de Cristoforis and Musolino \cite{LaMu13}).
Finally, by the real analyticity of $Z[\cdot]$ (cf.~Theorem \ref{ThmM}), one verifies that the map from $]-\epsilon_M,\epsilon_M[$ to $C^{1,\alpha}(\overline{\Omega_M})$ which takes $\epsilon$ to $\epsilon^{n-2} Z[\epsilon] \, S_n$ is real analytic.
Hence, one deduces the validity of $(ii)$.

Finally we prove $(iii)$. By \eqref{rappsol}, by Definition \ref{defstandardsolution}, by exploiting the homogeneity properties of $S_n$ and $\nabla S_n$, and by adding and subtracting the term $u^o(0)$, we obtain that
\begin{align*}
u^o_\epsilon(\epsilon t) = &\, u^o(0) + {u}^o(\epsilon t) - u^o(0) + \epsilon w^+_{\Omega^o}[\Phi^o[\epsilon] ](\epsilon t) 
\\&- \epsilon \int_{\partial\Omega^i} \nu_{\Omega^i}(s) \cdot \nabla S_n(t - s) \Phi^i[\epsilon](s) \,d\sigma_s 
+ \epsilon Z[\epsilon]\, S_n(t)  \quad \forall t \in \overline{\Omega_m}
\end{align*}
for all $\epsilon\in]0,\epsilon'[$. 
Then one observes that the map from $]-\epsilon_m,\epsilon_m[$ to $(C^{1,\alpha}(\overline{\Omega_m}))^n$ which takes $\epsilon$ to the function $\epsilon t$ of the variable $t$ is real analytic. Moreover, we have $\epsilon t \in \Omega^o$ for all $\epsilon \in ]-\epsilon_m,\epsilon_m[$ and all $ t \in \overline{\Omega_m}$. Then, by the real analicity of $u^o$ in $\Omega^o$ and known results on composition operators (cf.~Valent \cite[Thm.~5.2, p.~44]{Va88}), one verifies that the map from
\[
\{ h \in (C^{1,\alpha} (\overline{\Omega_m}))^n : h(\overline{\Omega_m}) \subset \Omega^o \}
\]
to $C^{1,\alpha}(\overline{\Omega_m})$ which takes a function $h$ to $u^o(h(\cdot))$ is real analytic. Hence, the map from $]-\epsilon_m,\epsilon_m[$ to $C^{1,\alpha}(\overline{\Omega_m})$ which takes $\epsilon$ to $u^o(\epsilon\cdot) - u^o(0)$ is real analytic and equal to $0$ for $\epsilon=0$. This implies that the map from $]-\epsilon_m,\epsilon_m[ \setminus \{0\}$ to $C^{1,\alpha}(\overline{\Omega_m})$ which takes $\epsilon$ to $\frac{u^o(\epsilon\cdot) - u^o(0)}{\epsilon}$ has a real analytic continuation to $]-\epsilon_m,\epsilon_m[$. Then it is natural to define
\begin{align*}
U^o_m[\epsilon](t) \equiv \frac{u^o(\epsilon\cdot) - u^o(0)}{\epsilon}  +  w^+_{\Omega^o}[\Phi^o[\epsilon] ](\epsilon t) - \int_{\partial\Omega^i} \nu_{\Omega^i}(s) \cdot \nabla S_n(t - s) \Phi^i[\epsilon](s) \,d\sigma_s 
\\+ Z[\epsilon]\, S_n(t)  \quad \forall t \in \overline{\Omega_m}
\end{align*}
for all $\epsilon\in]-\epsilon_m,\epsilon_m[$.

By the real analyticity of $\Phi^o[\cdot]$ (cf.~Theorem \ref{ThmM}) and by the properties of integral operators with real analytic kernels (see Lanza de Cristoforis and Musolino \cite{LaMu13}), it follows that the map from $]-\epsilon_m,\epsilon_m[$ to $C^{1,\alpha}(\overline{\Omega_m})$ which takes $\epsilon$ to $w^+_{\Omega^o}[\Phi^o[\epsilon] ](\epsilon \cdot)$ is real analytic.
Then, one considers the operator from $L^1(\partial \Omega^i)$ to $C^{1,\alpha}(\overline{\Omega_m})$ which takes $f$ to $\int_{\partial\Omega^i} \nu_{\Omega^i}(s) \cdot \nabla S_n(\cdot - s) f(s) \,d\sigma_s$.
By the real analyticity of $S_n$ on $\R^n\setminus\{0\}$, by the fact that the integral does not display singularities (by hypothesis $\Omega_m \subseteq \R^n\setminus\overline{\Omega^i}$), by the real analyticity of the map from $]-\epsilon_m,\epsilon_m[$ to $\COi_0$ which takes $\epsilon$ to $\Phi^i[\epsilon]$ (cf.~Theorem \ref{ThmM}) and since $\COi_0$ is linearly and continuously imbedded in $L^1(\partial \Omega^i)$, we conclude that the map from $]-\epsilon_m,\epsilon_m[$ to $C^{1,\alpha}(\overline{\Omega_m})$ which takes $\epsilon$ to $\int_{\partial\Omega^i} \nu_{\Omega^i}(s) \cdot \nabla S_n(\cdot -  s) \Phi^i[\epsilon](s) \,d\sigma_s $ is real analytic (see Lanza de Cristoforis and Musolino \cite{LaMu13}).
Finally, by the real analyticity of $Z[\cdot]$ (cf.~Theorem \ref{ThmM}), the map from $]-\epsilon_m,\epsilon_m[$ to $C^{1,\alpha}(\overline{\Omega_m})$ which takes $\epsilon$ to $Z[\epsilon] \, S_n$ is real analytic. Hence, one deduces the validity of $(iii)$.
\end{proof}


\section{Appendix}
In this appendix we present a technical result on the integration of real analytic maps in Banach spaces.

We will denote by $X'$ the space of continuous linear functionals from $X$ to $\mathbb{R}$, namely $X'=\mathcal{L}(X,\mathbb{R})$. Moreover, if $N \in \N \setminus\{0\}$, $X_1$, \dots, $X_N$ are Banach spaces, and $i_1$, \dots, $i_N$ are positive natural numbers, then $\mathcal{L}^{i_1,\dots,i_N}(X_1,\dots,X_N;X)$ will denote the space of continuous multilinear maps from  $X_1^{i_1}\times\dots\times X_N^{i_N}$ to $X$ endowed with the norm
\[
\|a\|_{\mathcal{L}^{i_1,\dots,i_N}(X_1,\dots,X_N;X)}=\sup_{Q}\|a[x_{1,1},\dots,x_{1,i_1},\dots,x_{N,1},\dots,x_{N,i_N}]\|_X
\]
where
\begin{align*}
Q=\{&(x_{1,1},\dots,x_{1,i_1},\dots,x_{N,1},\dots,x_{N,i_N})\in X_1^{i_1}\times\dots\times X_N^{i_N}\,:\;
\\ 
&\|x_{1,1}\|_{X_1}\le 1,\dots, \|x_{i,i_1}\|_{X_1}\le 1,
\dots,\|x_{N,1}\|_{X_N}\le 1,\dots,\|x_{N,i_N}\|_{X_N}\le 1\}.
\end{align*}
Finally, to shorten our notation, we set
\begin{equation*}
[x^{(i_1)}_1,\dots, x^{(j_N)}_N] = [\underbrace{x_1,\dots,x_1}_{i_1-\text{times}},\dots,\underbrace{x_N,\dots,x_N}_{i_N-\text{times}}]\,.
\end{equation*}

We now find convenient to recall the definition of real analytic maps from a Banach space $X$ to a Banach space $Y$ (see, for example, Deimling \cite{De85}) and the definition of Pettis integral in the case of maps from a bounded interval of $\R$ to a Banach space $X$ (see, for example, Pettis \cite{Pe38}).

\begin{defin}\label{anal}
	Let $X,Y$ be real Banach spaces. Let $U$ be an open subset of $X$.
	We say that a function $f$ from $U$ to $Y$ is real analytic if for every $x \in U$ there are $\rho,M \in ]0,+\infty[$ and multilinear maps $a_j(x) \in \mathcal{L}^j(X,Y)$, with $j \in \N$, such that
	\begin{equation*}
	\|a_j(x)\|_{\mathcal{L}^j(X,Y)} \leq M \left( \frac{1}{\rho}\right)^j \qquad \forall j \in \N
	\end{equation*}
	and
	\begin{equation*}
	f(y) = \sum_{j =0}^{+ \infty} a_j(x)[(y-x)^j] \qquad \forall y \in B_X(x,\rho).
	\end{equation*}
\end{defin}

\begin{defin}\label{Pettisint}
	Let $X$ be a Banach space and $a,b \in \R$. A function $F$ from $]a,b[$ to $X$ is said to be Pettis integrable over $]a,b[$ if there exists an element $x \in X$ such that
	\begin{equation*}
	L[x] = \int_{a}^{b} L[F(\tau)] \,d\tau \qquad \forall L \in X',
	\end{equation*}
	where the integral on the right hand side is the standard Lebesgue integral on $\R$.
	Then we define
	\begin{equation*}
	 \int_{a}^{b} F(\tau) \,d\tau \equiv x.
	\end{equation*}
\end{defin}

Our aim in this appendix is to prove the following theorem.

\begin{teo}\label{Sf(w)A(tw)}
	Let $X, Y$ be Banach spaces. Let $U$ be an open star-shaped subset of $X$ and let $A$ be a real analytic map from $U$ to $Y$. Let $f \in L^1([0,1])$. Then, for all $w\in U$ the integral 
	\begin{equation}\label{Sf(w)A(tw).eq1}
	\int_{0}^{1} f(\tau) A(\tau w) \,d\tau
	\end{equation} 
	 exists in the sense of Pettis and the map from $U$ to $Y$ which takes $w$ to \eqref{Sf(w)A(tw).eq1}  is real analytic.
\end{teo}

\begin{proof}
Since real analyticity is a local property, it suffices to prove the statement in a neighborhood of a fixed point $w^\ast$ of $U$. In the first part of the proof we introduce a suitable neighborhood.

{\bf Step 1.} We begin by observing that, being $U$ open and  star-shaped,  for every $\bar{\tau} \in [0,1]$ there exist $\delta_{\bar{\tau}} \in ]0,+\infty[$ and an open neighborhood $U_{\bar{\tau}}(w^\ast)\subset U$ of $w^\ast$ such that
	\begin{equation}\label{Sf(w)A(tw).eq2}
	\tau w \in U \qquad \forall (\tau,w) \in ]\bar{\tau} - \delta_{\bar{\tau}},\bar{\tau} 
	+ \delta_{\bar{\tau}} [ \times U_{\bar{\tau}}(w^\ast).
	\end{equation}
	Then, by the compactness of $[0,1]$, there exist $\tau_1,\dots,\tau_k \in [0,1]$ such that
	\begin{equation*}
	[0,1] \subset \bigcup_{j=1}^{k} ]\tau_j - \delta_{\tau_j},\tau_j 
	+ \delta_{\tau_j} [
	\end{equation*}
	with $\delta_{\tau_j}$ as in \eqref{Sf(w)A(tw).eq2} for all $j \in \{1,\dots,k\}$.	
	If we now define
	\begin{equation*}
	U(w^\ast) \equiv \bigcap_{j=1}^k U_{\tau_j}(w^\ast) \quad \mbox{and} \quad I(w^\ast) \equiv \bigcup_{j=1}^{k} ]\tau_j - \delta_{\tau_j},\tau_j 
	+ \delta_{\tau_j} [\,,
	\end{equation*}
	then we have that $U(w^\ast)$ and $I(w^\ast)$ are open, that
	\begin{equation}\label{[0,1]inI}
	[0,1]\subset I(w^*)\,,
	\end{equation}
	and that
	\[
	\tau w \in U\qquad\forall (\tau,w) \in I(w^\ast)\times U(w^\ast)\,.
	\]
	As a consequence,  the map from $I(w^\ast) \times U(w^\ast)$ to $U$ which takes $(\tau,w)$ to $\tau w$ is well defined and, being bilinear and continuous,  it is also real analytic.  It follows that the map from $I(w^\ast) \times U(w^\ast)$ to $Y$ which takes $(\tau,w)$ to $A(\tau w)$ is  real analytic, being the composition of real analytic maps. By Definition \ref{anal} of real analytic maps we deduce that, for all fixed $\tau' \in I(w^\ast)$, there exist positive real numbers $M(\tau',w^\ast)$ and $\rho(\tau',w^\ast)$ and a family of multilinear maps $\{a_{ij}(\tau',w^\ast)\}_{i,j \in \N} \subset \mathcal{L}^{i,j}(\R,X;Y)$  such that
	\begin{equation*}
	\|a_{ij}(\tau',w^\ast)\|_{\mathcal{L}^{i,j}(\R,X;Y)} \leq M(\tau',w^\ast) \left(\frac{1}{\rho(\tau',w^\ast)}\right)^{i+j} \qquad \forall i,j \in \N
	\end{equation*}
	and such that
	\begin{equation*}
	A(\tau w) = \sum_{i,j=0}^{\infty} a_{ij}(\tau',w^\ast) [(\tau-\tau')^{(i)},(w-w^\ast)^{(j)}]
	\end{equation*}
	for all $(\tau,w) \in ]\tau'- \rho(\tau',w^\ast),\tau'+ \rho(\tau',w^\ast)[\, \times B_X(w^\ast, \rho(\tau',w^\ast))$. Moreover, since the first $i$ arguments of  the $a_{i,j}(\tau',w^\ast)$'s are real, one verifies that there are multilinear maps $b_{i,j}(\tau',w^\ast)\in \mathcal{L}^j(X;Y)$ such that
	\[
	a_{i,j}(\tau',w^\ast)[(\tau-\tau')^{(i)},(w-w^\ast)^{(j)}]=(\tau-\tau')^i \, b_{i,j}(\tau',w^\ast)[(w-w^\ast)^{(j)}]\qquad\forall i,j \in \N\,.
	\]
Then we have 
	\begin{equation}\label{bij(tau,w)inequality}
	\|b_{ij}(\tau',w^\ast)\|_{\mathcal{L}^j(X;Y)} \leq M(\tau',w^\ast) \left(\frac{1}{\rho(\tau',w^\ast)}\right)^{i+j} \qquad \forall i,j \in \N
	\end{equation}
	and
	\begin{equation}\label{A(tau w)}
	A(\tau w) = \sum_{i,j=0}^{\infty} (\tau-\tau')^i \, b_{ij}(\tau',w^\ast) [(w-w^\ast)^{(j)}]
	\end{equation}
	where the series converges absolutely and uniformly for $(\tau,w)$ in $]\tau'- \rho(\tau',w^\ast),\tau'+ \rho(\tau',w^\ast)[ \, \times B_X(w^\ast, \rho(\tau',w^\ast))$. 
	We now observe that the set $\{B(\tau', {\rho(\tau',w^\ast)}/{2} )\,:\;\tau' \in I(w^\ast)\}$ is an open covering of $[0,1]$ (cf.~\eqref{[0,1]inI}). Then, by a standard compactness argument it follows that there exist $\tau'_1,\dots,\tau'_h \in [0,1]$ and disjoint  intervals  $I_1,\dots,I_h\subset [0,1]$ such that $I_1\cup \dots \cup I_h=[0,1]$ and 
	\begin{equation}\label{I_lconteinedB(tau,rho/2)}
	I_l \subset B\left(\tau'_l, \frac{\rho(\tau'_l,w^\ast)}{2} \right) \qquad \forall l \in \{1,\dots,h\}
	\end{equation}
	(some of the $I_l$'s might be empty). 	Finally, we define 
	\begin{equation}\label{rho(w*)}
	\rho(w^\ast) \equiv \min_{l \in 1,\dots,h} \rho(\tau'_l,w^\ast)\,.
	\end{equation}
	In the next step of the proof we show that the statement of the theorem holds in $B\left(w^\ast,\frac{\rho(w^\ast)}{2}\right)$.
	To do so, we also find convenient to set
	\begin{equation}\label{M(w*)}
	M(w^\ast) \equiv \max_{l=1,\dots,h} M(\tau'_l,w^\ast).
	\end{equation}

{\bf Step 2.}	We claim that for all $w \in B\left(w^\ast,\frac{\rho(w^\ast)}{2}\right)$   the Pettis integral
\[
\int_{0}^{1} f(\tau) A(\tau w) \,d\tau
\]
is given by the sum
\begin{equation}\label{sum}
 \sum_{l=1}^{h}\sum_{i,j=0}^{\infty}  \left(\int_{I_l}  f(\tau)(\tau-\tau'_l)^i \,d\tau \right)  b_{ij}(\tau'_l,w^\ast) [(w-w^\ast)^{(j)}]\,.
\end{equation}
To prove it,  we first verify that \eqref{sum} defines an element of $Y$. Indeed, if  $w \in B\left(w^\ast,\frac{\rho(w^\ast)}{2}\right)$, then \eqref{bij(tau,w)inequality}, \eqref{rho(w*)}, and \eqref{M(w*)} imply that
\begin{equation}\label{|bij[w]|}
    \begin{aligned}
    &\left\|b_{ij}(\tau'_l,w^\ast) [(w-w^\ast)^{(j)}] \right\|_Y \leq M(\tau'_l,w^\ast) \left( \frac{1}{\rho(\tau'_l,w^\ast)} \right)^{i+j} \left( \frac{\rho(w^\ast)}{2}\right)^j
    \\
    &\quad\leq   \left( \frac{1}{2}\right)^{j}\, M(\tau'_l,w^\ast) \left( \frac{1}{\rho(\tau'_l,w^\ast)} \right)^{i}\left( \frac{\rho(w^\ast)}{\rho(\tau'_l,w^\ast)} \right)^{j} \leq\left( \frac{1}{2}\right)^{j}\, M(w^\ast) \left( \frac{1}{\rho(\tau'_l,w^\ast)} \right)^{i}    
    \end{aligned}
\end{equation}
for all $i,j \in \N$. Hence, by \eqref{I_lconteinedB(tau,rho/2)} and \eqref{|bij[w]|} we have
\[
\begin{split}
&\left\| \left(\int_{I_l}  f(\tau)(\tau-\tau'_l)^i \,d\tau \right)  b_{ij}(\tau'_l,w^\ast) [(w-w^\ast)^{(j)}] \right\|_Y
\\
&\qquad\le \| f \|_{L^1([0,1])} \left( \frac{\rho(\tau'_l,w^\ast)}{2} \right)^i\left\|  b_{ij}(\tau'_l,w^\ast) [(w-w^\ast)^{(j)}] \right\|_Y\\
&\qquad\le \left( \frac{1}{2}\right)^{i+j} \| f \|_{L^1([0,1])}  M(w^\ast).
\end{split}
\]
The last inequality readily implies the convergence in $Y$ of the series in \eqref{sum}.
	
In view of Definition \ref{Pettisint} of  Pettis integral, we now
consider a functional $L \in Y'$ and we observe that for all fixed $w \in B\left(w^\ast,\frac{\rho(w^\ast)}{2}\right)$ the function which takes $\tau\in ]0,1[$ to $L[A(\tau w)]$ is continuous. Since $f\in L^1([0,1])$, it follows that  the function which takes $\tau\in]0,1[$ to
	\[
	L[f(\tau)A(\tau w)]=f(\tau)L[A(\tau w)]
	\]
         belongs to $L^1([0,1])$.  Then, by splitting the integral on $\tau\in]0,1[$ over the partition $I_1$,\dots,$I_h$, by the uniform convergence of the series in \eqref{A(tau w)}, and by \eqref{I_lconteinedB(tau,rho/2)} we obtain that
	\begin{equation}\label{integral01<L,f(w)A(tau w)>}
	\begin{aligned}
	\int_{0}^{1} L[f(\tau)A(\tau w)] \,d\tau &= \sum_{l=1}^{h} \int_{I_l} L[f(\tau)A(\tau w)] \,d\tau
	\\
	&=\sum_{l=1}^{h} \int_{I_l} L\biggl[f(\tau)  \sum_{i,j=0}^{\infty} (\tau-\tau'_l)^i \, b_{ij}(\tau'_l,w^\ast) [(w-w^\ast)^{(j)}] \biggr] \,d\tau
	\\
	&=\sum_{l=1}^{h} \int_{I_l} \sum_{i,j=0}^{\infty} L\left[f(\tau)(\tau-\tau'_l)^i \, b_{ij}(\tau'_l,w^\ast) [(w-w^\ast)^{(j)}]\right] \,d\tau 
	\\
	&=\sum_{l=1}^{h} \int_{I_l}  \sum_{i,j=0}^{\infty}f(\tau) (\tau-\tau'_l)^i L\left[ b_{ij}(\tau'_l,w^\ast) [(w-w^\ast)^{(j)}] \right] \,d\tau\,. 
	\end{aligned}
	\end{equation}
	To verify that the Pettis integral of $f(\tau)A(\tau w)$ on $[0,1]$ is given by \eqref{sum},  it  remains  to show that we can  change the order of the integration  over $I_l$  and of the summation on $i,j$ in \eqref{integral01<L,f(w)A(tau w)>}. By a classical corollary of the dominated convergent theorem it suffices to prove that
	\[
	\sum_{i,j=0}^{\infty}\int_{I_l} \left| f(\tau) (\tau-\tau'_l)^i L\left[ b_{ij}(\tau'_l,w^\ast) [(w-w^\ast)^{(j)}] \right]\right| \,d\tau
	\]
	is a convergent series. This latter fact can be deduced by noting that, as a consequence of \eqref{|bij[w]|}, we have
\begin{equation*}
    \begin{aligned}
    &\left|(\tau-\tau'_l)^i L\left[ b_{ij}(\tau'_l,w^\ast) [(w-w^\ast)^{(j)}] \right] \right| 
    \\
    &\quad\leq  \left( \frac{\rho(\tau'_l,w^\ast)}{2} \right)^i \|L\|_{Y'} \, \left( \frac{1}{2}\right)^{j}\, M(w^\ast) \left( \frac{1}{\rho(\tau'_l,w^\ast)} \right)^{i}\leq \left( \frac{1}{2}\right)^{i+j} \|L\|_{Y'} \, M(w^\ast) 
    \end{aligned}
    \end{equation*} 
    for all $i,j \in\N$, $l \in \{1,\dots,h\}$,  and $\tau\in I_l$. 
    
Now that we know that the  integral $\int_{0}^{1} f(\tau) A(\tau w) \,d\tau$ is given by \eqref{sum}, the real analyticity of the map that takes $w$ to $\int_{0}^{1} f(\tau) A(\tau w) \,d\tau$ is a direct consequence of Definition \ref{anal} of real analytic maps. 
\end{proof}

\section*{Acknowledgments}

The research of the author was supported by HORIZON 2020 RISE project "MATRIXASSAY" under project number 644175, during the author's secondment at the University of Texas at Dallas. The author gratefully acknowledges the University of Texas at Dallas and the University of Tulsa for the great research environment and the friendly atmosphere provided. The author is indebted to Matteo Dalla Riva, Gennady Mishuris and Paolo Musolino for useful discussions and comments on the problem studied in this paper.


\vspace{\baselineskip}

{\sc Riccardo Molinarolo\\
	Department of Mathematics, IMPACS,
	Aberystwyth University, Aberystwyth,
	Ceredigion SY23 3BZ, UK }

{\it E-mail address:} \verb+rim22@aber.ac.uk+

\end{document}